\documentclass[11pt]{amsart}

\usepackage{amsmath,amssymb,amsxtra,mathrsfs,latexsym}
\usepackage{graphicx}
\usepackage[all]{xy}
\usepackage{float}
\usepackage{verbatim}
\usepackage{adjustbox}
\usepackage{url}

\usepackage{tikz}
\usetikzlibrary{shapes.misc, positioning}
\usepackage{tikz-cd} 

\usepackage{todonotes}

\usepackage[colorlinks=true, linkcolor=blue, citecolor=blue, urlcolor=blue, filecolor=blue]{hyperref}

\newtheorem{theorem}{Theorem}[section]
\newtheorem{claim}[theorem]{Claim}
\newtheorem{setup}{Setup}

\newtheorem{lemma}[theorem]{Lemma}

\newtheorem{proposition}[theorem]{Proposition}
\newtheorem{corollary}[theorem]{Corollary}

\theoremstyle{definition}
\newtheorem{definition}[theorem]{Definition}

\newtheorem{question}[theorem]{Question}

\theoremstyle{remark}
\newtheorem{remark}[theorem]{Remark}

\newtheorem{notation}[theorem]{Notation}
\newtheorem{fact}[theorem]{Fact}

\newcommand\cat[1]{{}^\curvearrowright #1}

\DeclareMathOperator{\cof}{cof}

\DeclareMathOperator{\crit}{crit}

\DeclareMathOperator{\dom}{dom}

\DeclareMathOperator{\HOD}{HOD}
\DeclareMathOperator{\OD}{OD}
\DeclareMathOperator{\ob}{OB}

\DeclareMathOperator{\Add}{Add}

\DeclareMathOperator{\id}{id}
\newcommand{\cf}{{\rm cf}}

\def\s{\subseteq}

\def\forces{\Vdash}
\def\br{\blacktriangleright}
\newcommand{\one}{\mathop{1\hskip-3pt {\rm l}}}

\newcount\skewfactor
\def\mathunderaccent#1#2 {\let\theaccent#1\skewfactor#2
  \mathpalette\putaccentunder}
\def\putaccentunder#1#2{\oalign{$#1#2$\crcr\hidewidth
  \vbox to.2ex{\hbox{$#1\skew\skewfactor\theaccent{}$}\vss}\hidewidth}}

\def\smallbox#1{\leavevmode\thinspace\hbox{\vrule\vtop{\vbox
   {\hrule\kern1pt\hbox{\vphantom{\tt/}\thinspace{\tt#1}\thinspace}}
   \kern1pt\hrule}\vrule}\thinspace}

\title[Compactness phenomena and  Magidor's covering theorem]{Compactness phenomena in HOD and the Optimality of Magidor's Covering theorem}

\author[Benhamou]{Tom Benhamou}
\address[Benhamou]{School of Mathematical Sciences, Tel Aviv University, Tel Aviv 69978, Israel}
\email{tombenhamou@tauex.tau.ac.il}
\urladdr{https://www.math.tau.ac.il/~tombenhamou/}

\author[Cummings]{James Cummings}
\address[Cummings]{Department of Mathematical Sciences, Carnegie Mellon University, Pittsburgh, PA 15213, USA}
\email{jcumming@andrew.cmu.edu}

\author[Goldberg]{Gabriel Goldberg}
\address[Goldberg]{Department of Mathematics, UC Berkeley, CA 94720, USA}
\email{ggoldberg@berkeley.edu}

\author[Hayut]{Yair Hayut}
\address[Hayut]{Einstein Institute of Mathematics, Hebrew University of Jerusalem, Givat-Ram, 91904, Israel.}
\email{yair.hayut@mail.huji.ac.il}

\author[Poveda]{Alejandro Poveda}
\address[Poveda]{Departamento de Matemáticas, CUNEF Universidad,  Madrid, 28040, Spain \textbf{and} 
Fachbereich Mathematik, Universität Hamburg,
Hamburg, 20146, Germany.}
\email{alejandro.poveda@cunef.edu}
\urladdr{https://www.alejandropovedaruzafa.com}
\thanks{Poveda presented the results of this paper at the Simons Semester \emph{``Gödel's Program''} held at IMPAN, as well as at the Set Theory seminar of the Hebrew University of Jerusalem. The authors are very grateful to all the participants for their comments and constructive feedback. }


\begin{document}

\maketitle

\begin{abstract}
 We investigate compactness phenomena around the power-set functions of $V$ and $\HOD$. We prove: (1) A singular strong limit cardinal with uncountable cofinality cannot be the first place where $\mathcal{P}(\cdot )$ and $ \mathcal{P}^{\HOD}(\cdot)$ disagree. (2) Assuming the existence of a measurable cardinal, $\aleph_\omega$ can be the first place where $\mathcal{P}(\aleph_\omega)\neq \mathcal{P}^{\HOD}(\aleph_\omega)$, answering a question of Hayut. (3) If $\kappa$ is strong limit singular of uncountable cofinality, $\kappa$ cannot be the first place where $(\HOD, V)$ fails to have the cover property; (4)   
Magidor's  Covering Theorem is optimal.
\end{abstract}

\section{Introduction}
 
The present manuscript is concerned with compactness phenomena within Gödel's universe of \emph{Hereditarily Ordinal Definable} sets ($\HOD$) \cite{Godel}. Compactness is the phenomenon whereby local properties of a mathematical structure (e.g., a graph) determine global properties of that structure. To illustrate this principle, consider the classical theorem of Erd"os--de Bruijn \cite{DeBrujin}, which asserts that a loopless planar graph $\mathcal{G}$ has chromatic number at most $k\in\mathbb{N}$ provided that every finite subgraph $\mathcal{H}\leq\mathcal{G}$ has chromatic number at most $k$. Combined with the \emph{Four Color Theorem}, this compactness result implies that every such graph $\mathcal{G}$, including infinite ones, has chromatic number at most $4$. This is but one instance of a ubiquitous phenomenon across mathematics, with manifestations in areas such as Functional Analysis \cite{Abad,CervantesPoveda}, Algebra \cite{MagidorShelah, EklofMekler, CortesPoveda, CoxPoveda}, Topology \cite{BagMag}, and Category Theory \cite{Bagetal}, among many others.
 
\smallskip
 
Set theory, the branch of mathematics devoted to the study of the infinite, has a
long-standing tradition of investigating compactness phenomena
\cite{MitchellAronszajn,MagRefl,MR924672,CumFor,CFM,PartIII,GolPov,JakobPoveda,BenhamouSinapova,HayutPoveda}.
A classical result in this vein, which has inspired much subsequent work, is
\emph{Silver's theorem} \cite{MR0429564}: a singular cardinal of uncountable cofinality
cannot be the first place where Cantor's \emph{Generalized Continuum Hypothesis} (GCH)
fails. For instance, if $2^{\aleph_\alpha}=\aleph_{\alpha+1}$ for all $\alpha<\omega_1$,
then $2^{\aleph_{\omega_1}}=\aleph_{\omega_1+1}$ as well.
This phenomenon is exclusive to singular cardinals of uncountable cofinality. At regular cardinals the continuum function $\kappa\mapsto2^{\kappa}$
is essentially unconstrained: Gödel \cite{GodelL} and Cohen showed that the value of $2^{\aleph_0}$ is independent
of ZFC \cite{cohen1,cohen2}, and Easton \cite{Easton} that the restriction of this function
to the regular cardinals can be prescribed almost arbitrarily. At singular cardinals of
countable cofinality, Magidor \cite{MagSingII} showed that,
modulo \emph{large cardinals}, it is consistent that GCH first fails at $\aleph_\omega$.
Together, these results reveal a foundational difference between singular cardinals of
uncountable cofinality and all other infinite cardinals.

 \smallskip

 In a 1946 Princeton lecture \cite{Godel}, G\"odel proposed the notion of \emph{ordinal definability} as a means to circumvent the paradoxes arising from formalizing \emph{mathematical definability}. Intuitively, a set $X$ is \emph{ordinal definable} if it is the unique object satisfying some mathematical property that can be defined in terms of finitely many ordinal numbers. A set $X$ is \emph{hereditarily ordinal definable} ($X\in \HOD$) if $X$ is ordinal definable, every element of $X$ is ordinal definable, every element of every element of $X$ is ordinal definable, and so on; see \S\ref{sec: preliminaries} for a  formal definition. After Gödel, subsequent investigation by Myhill--Scott \cite{Myhill}, Vop\v{e}nka \cite{Vopenka}, and others began to outline the theory of HOD. In last few years, the study of HOD has experienced renewed interest after groundbreaking findings due to  Woodin \cite{WooPartI,WoodinPartII,midrasha,BKW} in connection to \emph{Gödel's Program}; the project aimed at finding the ultimate foundation of mathematics.
 
\smallskip
 
A wealth of work has emerged from Woodin's insights into HOD, much of it devoted to
comparing the combinatorial structure of $\HOD$ with that of the set-theoretic universe $V$.
To name a few highlights: Cummings--Friedman--Golshani \cite{CumFriGol} showed the
consistency of ZFC with $\HOD$ failing to compute the successor of any cardinal;
Cummings et al.\ \cite{AIM} established consistency of successors of singular cardinals of
countable cofinality being large cardinals in $\HOD_{\{x\}}$ for all $x \subseteq \kappa$;
Ben-Neria--Hayut \cite{BenHay} proved consistency of every regular cardinal's successor
being $\omega$-strongly measurable in $\HOD$, which was recently upgraded to all regular cardinals by  Blue--Larson--Sargsyan \cite{BlueLarsonSargsyan}; Goldberg \cite{Gol} provided a detailed
analysis of cover properties between $\HOD$ and $V$ under the HOD hypothesis, which Poveda
\cite{PovOmega} showed is essentially optimal; and most recently, Goldberg--Poveda
 established optimality of the HOD dichotomy \cite{GOP} and initiated the study of
compactness phenomena between $\HOD$ and $V$ \cite{GolPov}. In this last paper, they proved that if
$\kappa$ is a strong limit singular cardinal of uncountable cofinality, then stationarity
below $\kappa$ of either $\{\lambda<\kappa\mid \cf^{\HOD}(\lambda)<\lambda\}$ or
$\{\lambda<\kappa\mid \lambda^{+\HOD}=\lambda^+\}$ implies the corresponding property holds at
$\kappa$ itself. These results echo Silver's theorem.  
\smallskip

In the first part of the manuscript, we continue the study initiated in \cite{GolPov} and investigate compactness phenomena concerning the power-set function. More precisely, we study the \emph{first place of disagreement} between $\HOD$ and the set-theoretic universe $V$. The first place of disagreement between $\HOD$ and $V$, denoted by $\Delta(\HOD,V)$, is defined to be the least ordinal $\kappa$ such that $\mathcal{P}(\kappa)\nsubseteq\HOD$, provided $V\neq\HOD$. It is easy to see that any regular cardinal $\kappa$ can be forced to be $\Delta(\HOD,V)$. Moreover, if $\kappa$ possesses a mild large cardinal property, such as Mahloness, one can arrange that $\Delta(\HOD,V)=\kappa$ while preserving the large cardinal property of $\kappa$.

\smallskip

The situation with singular cardinals is, as usual,  much more subtle. In 2015, Hayut asked whether the first singular cardinal, $\aleph_\omega$, can be the first place of disagreement between $\HOD$ and $V$. 
We show that this is possible starting from the optimal large cardinal assumptions:
 
\begin{theorem}\label{thm: alephomega}
    Assuming the existence of a measurable cardinal, there is a  generic extension where $\aleph_\omega$ is strong limit and $$\Delta(\HOD, V)=\aleph_\omega.$$
    Moreover, the large cardinal assumptions used are optimal.
\end{theorem}
For singular cardinals of uncountable cofinality, the situation is far more restrictive:
ZFC proves a compactness result in the spirit of Silver's theorem. More precisely, we have
the following:
 
\begin{theorem}\label{thm: compactness}
 Suppose that $\kappa$ is a singular cardinal with $\cf(\kappa)\geq \omega_1$ such that $2^{\cf(\kappa)}<\kappa$. Then, $\Delta(\HOD, V)\neq \kappa$.
\end{theorem}
The same argument provides a variant of a theorem of Shelah  \cite{ShelahHOD} giving the existence of an $\mathrm{OD}_x$ well-ordering of $\mathcal{P}(\kappa)$ for a certain $x\s \kappa$:
\begin{theorem}\label{thm: shelah variant}
 Let $\kappa$ be a singular cardinal with $\cf(\kappa)\geq \omega_1$ such that $2^{\cf(\kappa)}<\kappa$. Suppose that there is an $\mathrm{OD}$ well-ordering of $\mathcal{P}(\alpha)$ for all $\alpha < \kappa$. Then, there is an $\mathrm{OD}$ well-ordering of $\mathcal{P}(\kappa)$.    
\end{theorem}
The key to the above results is the following spin-off of a theorem of Goldberg--Poveda from \cite{GolPov}:
\begin{theorem}\label{thm: spinoffGP}
 Suppose that $\kappa$ is a singular cardinal of uncountable cofinality such that $\cf(\lambda)=\cf^{\HOD}(\lambda)$ for all $\lambda < \kappa$. Then, $\cf(\kappa)=\cf^{\HOD}(\kappa)$.   
\end{theorem}
 
In the second part of this paper we  investigate compactness phenomena related to the cover property of the pair $(\HOD, V)$. Given a cardinal $\kappa$, we say that \emph{$(\HOD, V)$ has the cover property at $\kappa$} if for  every set $X \s \kappa$ there is $Y\in \HOD$ such that $X\s Y$ and $|Y|=\max\{|X|,\aleph_1\}$. We are interested in understanding the first cardinal $\kappa$ for which the pair $(\HOD, V)$ fails to satisfy the cover property at $\kappa$. {For regular cardinals $\kappa$ it is not hard to arrange situations where $\kappa$ is the first place where $(\HOD, V)$ fails to have the cover property at $\kappa$.} The same comment applies to singular cardinals of countable cofinality utilizing the standard Prikry forcing.  However, the picture with singular cardinals of uncountable cofinality is much more subtle. 

\begin{theorem}\label{cor:fullcovering}
   A singular strong limit cardinal $\kappa$ with $\cf(\kappa)\geq \omega_1$ cannot be the first cardinal where $(\HOD, V)$ fails to satisfy the cover property at $\kappa$.
\end{theorem}
Assuming that  $(\HOD, V)$ computes cardinals  below $\kappa$ alike it is possible to prove a version of the above theorem where the covering assumptions of $(\HOD, V)$ at every  $\lambda < \kappa$ are weakened to $\cf(\lambda)^+$-cover at $\lambda$:\footnote{The pair $(\HOD, V)$ is said to have the $\mu$-cover property at $\kappa$ if for every set $X\s \kappa$ there is $Y\in \HOD$ such that $|Y|<\mu$ and $X\s Y$.}

\begin{theorem}\label{thm:cflambda+cover}
Suppose that $\kappa$ is a strong limit cardinal with $\cf(\kappa)\geq \omega_1$ and the following properties hold:
\begin{enumerate}
    \item $\mathrm{Card}^{\HOD}\cap \kappa=\mathrm{Card}\cap \kappa$.
    \item $(\HOD, V)$ has the  $\cf(\lambda)^+$-cover property at $\lambda$ for all  $\lambda < \kappa$.\footnote{If $\lambda$ is regular then the assumption is trivially true.}
\end{enumerate}
Then, $(\HOD, V)$ has the full cover property at every ordinal $\alpha \leq \kappa$. 
\end{theorem}

What happens if the inner model is not assumed to be $\HOD$? In this case, under  certain additional arithmetic assumptions on  $M$ one can proof a similar theorem to Theorem~\ref{thm:cflambda+cover}. The proof (see Proposition~\ref{lemma: local Magidor}) builds upon ideas of Magidor,  who established the following elegant result: If $M$ is a cofinality--correct inner model, $M\models \mathrm{GCH}$ and $(M,V)$ has the $\omega_1$-cover property then $(M, V)$ has the full cover property (\cite[Theorem~2.33]{AbrahamMagidor}). 

In his Ph.D. dissertation \cite{SebaPhD}, S. Thei observed that Shelah's weakening of GCH known as the \emph{Singular Strong Hypothesis} (SSH)\footnote{SSH is the postulate asserting that $``\mathrm{pp}(\kappa)=\kappa^+$ holds for all singular cardinals $\kappa$", where $\mathrm{pp}(\kappa)$ denotes the \emph{pseudopower of $\kappa$}. Since we will not make use of pseudopowers here, we refrain from defining this notion and refer the reader to Shelah's book \cite{ShelahBook} for the relevant background.} suffices to establish the conclusion of Magidor's  covering theorem.  In light of the above, it is natural to ask the  extent to which  SSH assumption  is necessary.  We show that this is the case, by proving the following consistency result:
 
\begin{theorem}\label{thm:optimality of Magidor}
Assuming the existence of appropriate large cardinals, the $\mathrm{GCH}$ assumption cannot be removed from Magidor's covering theorem. 
\end{theorem}
 The precise statement of this theorem is rather technical -- it can be found in Theorem~\ref{thm:optimality of Magidor}. The theorem, in fact, provides an optimal failure of Magidor's theorem by supplying a pair of inner models $(M,V)$ that agree up to a singular cardinal $\kappa$ of uncountable cofinality but $(M,V)$ fails to have the $\cf(\kappa)^+$-cover property at $\kappa$.
 
\smallskip
 
The structure of the paper is as follows. In Section~\ref{sec: preliminaries} we provide a few preliminaries. In Section~\ref{sec: compactness HOD} we discuss the first point of disagreement between $\HOD$ and $V$. In Section~\ref{sec: compactness of covering} we prove the compactness of the covering property for the pair $(\HOD, V)$. In Section~\ref{sec: optimality of Magidors} we prove the optimality of Magidor's covering theorem. Finally, Section~\ref{sec: open questions} features a few open questions.

\section{Preliminaries}\label{sec: coding}\label{sec: preliminaries}

The primary object of study in this manuscript is Gödel's universe of Hereditary Ordinal Definable sets, $\HOD$ (see \cite{Godel, Myhill}):
\begin{definition}
    Let $X$  be a set. The class of  \emph{Ordinal Definable sets relative to $X$}, denoted $\OD_X$,  consists of all sets $A$ for which there is a formula $\phi(z, x_0,\dots, x_{m-1},y_0\dots, y_{n-1})$ in the language of set theory, sets $a_0,\dots, a_{m-1}$ in  $X$ and ordinals $\beta, \alpha_0,\dots, \alpha_{n-1}$ such that $A$ can be presented as $$A=\{z\in V_\beta \mid V_\beta \models \varphi(z,a_0,\dots,a_{m-1},\alpha_0,\dots, \alpha_{n-1})\}.$$
    The class of \emph{Hereditarily Ordinal Definable sets relative to $X$}, denoted $\HOD_{\{x\}}$, is the class of all $A\in \mathrm{OD}_X$ such that $\mathrm{trcl}(\{A\})\s \OD_X$.

   We shall write $\OD$ and $\HOD$ for $\OD_\emptyset$ and $\HOD_\emptyset$, respectively.
\end{definition}
$\HOD$ is an \emph{inner model}; namely, a transitive $\in$-structure satisfying the ZFC axioms and containing all the ordinals (through the paper this is what we shall mean when we refer to an \emph{inner model}).

\smallskip

The main interest of this paper is to analyze how \emph{set-theoretic forcing} affects the properties of $\HOD$. Our main reference for the method of forcing is Kunen's book \cite{Kunen}. 
 Our notation for (forcing) posets will be $\mathbb{P},  \mathbb{Q}$, etc. Members $p$ of a poset $\mathbb{P}$ will be referred as \emph{conditions}. Given two conditions $p,q\in \mathbb{P}$ we will write $``q\leq p$'' as a shorthand for \emph{$``q$ is stronger than $p$''}. We shall denote by $\mathbb{P}/p$ the subposet of $\mathbb{P}$ whose universe is  $\{q\in \mathbb{P}\mid q\leq p\}$. Two conditions $p$ and $q$ are \emph{compatible} if there is a condition $r\in \mathbb{P}$ such that $r\leq p,q$. A set $D\s \mathbb{P}$ is called \emph{dense} if for each $p\in \mathbb{P}$ there is $q\in D$ such that $q\leq p$. A set $G\s \mathbb{P}$ is called a \emph{generic filter} if every two members of $G$ are compatible, $G$ is upwards-closed\footnote{I.e., if $p\in \mathbb{P}$ and $q\in G$ are such that $q\leq p$ then $p\in G$.}, and $G$ intersects every dense set. Central to this paper are projections and isomorphisms:

    \begin{definition}\label{def: projections}
        Let $\mathbb{P}$ and $\mathbb{Q}$ be forcing posets. \begin{itemize}
\item  A \emph{projection} is a map $\pi\colon \mathbb{P}\rightarrow\mathbb{Q}$ such that: \begin{enumerate}
\item $\pi(\one_{\mathbb{P}})=\one_{\mathbb{Q}}$.\footnote{Recall that the weakest condition of a forcing poset $\mathbb{P}$ is customarily denoted by $\one_{\mathbb{P}}$ -- or simply by $\one$ if there is no confusion. }
            \item For all $p,p'\in\mathbb{P}$ with $p\leq p'$, $\pi(p)\leq \pi(p')$;
            \item {For all $p\in\mathbb{P}$ and $q\in \mathbb{Q}$ with $q\leq \pi(p)$ there is $p'\leq p$  such that $\pi(p')\leq q$.}
        \end{enumerate} 
        
            \item We say that $\mathbb{P}$ and $\mathbb{Q}$ are  \emph{isomorphic} if there are maps $\phi\colon \mathbb{P}\to \mathbb{Q}$, $\psi\colon \mathbb{Q}\to \mathbb{P}$ such that $\phi\circ \psi=\id_{\mathbb{Q}}$ and $\psi\circ \phi=\id_{\mathbb{P}}$ and both $\phi$ and $\psi$ are order-preserving.
        \end{itemize} 
    \end{definition}

        Let $\mathbb{P}$ and $\mathbb{Q}$ be forcing notions and $G$ be $\mathbb{P}$-generic. If $\pi\colon \mathbb{P}\rightarrow\mathbb{Q}$ {is a projection}, then the upwards closure of the set $\pi``G$
            is a $\mathbb{Q}$-generic filter. For further information we refer to Abraham's Handbook chapter \cite{Abra}.

            \smallskip
            
Another concept used in this paper is due to Neeman \cite{Neeman2014}:
\begin{definition}
    Let $\mathbb{P}$ be a poset and $M$ be a set. A  function $f\colon \mathbb{P}\to \mathbb{P}\cap M$ is called a \emph{residue} if for any $p\in \dom(f)$,  every $q\leq f(p)$ with $q\in \mathbb{P}\cap M$ is compatible with $p$.
\end{definition}
The interest of this notion hinges on the following standard fact, whose proof we include for the benefit of the reader:
\begin{proposition}\label{cor: submodel}
Suppose that $f\colon \mathbb{P}\to \mathbb{P}\cap M$ is a residue. Then,  $\mathbb{P}\cap M$ is a complete subforcing of $\mathbb{P}$.\footnote{That is, any maximal antichain of $\mathbb{P}_M$ is also a maximal antichain of $\mathbb{P}$.} In particular, if $\Theta$ is a sufficiently large regular cardinal, $M\prec H_\Theta$ and $\mathbb{P}\in M$, $G\cap M$ is generic for $\mathbb{P}\cap M$.
\end{proposition}
\begin{proof}
Fix $A\s \mathbb{P}\cap M$ a maximal antichain and let  $p\in \mathbb{P}$ be arbitrary. Since $A$ is maximal in $\mathbb{P}\cap M$, there is $q\in A$ that is $\mathbb{P}\cap M$-compatible with $f(p)$. Let $r\in \mathbb{P}\cap M$ be such that $r\leq f(p)$, $q$. Since $f$ is a residue  we can find $p'\in \mathbb{P}$ such that $p'\leq r, p$. In particular, $p'$ witnesses the compatibility of $p$ and $q$, ergo $A$ is a maximal antichain in $\mathbb{P}$ as well.

To establish the second claim we assume that $M\prec H_\Theta$. Under this assumption, any two conditions in  $G\cap M$ are compatible in $\mathbb{P}\cap M$, and $G\cap M$ is also upwards closed in this latter poset (i.e., $G\cap M$ is a filter). The fact that $G\cap M$ intersects all the maximal antichains of $\mathbb{P}\cap M$ (i.e.  $\mathbb{P}\cap M$-genericity) was already established in the previous argument.
\end{proof}
            
At many places in this paper we will be interested in controlling $\HOD$ of a generic extension. The concept that makes possible this analysis is the following:

    \begin{definition}
        A poset $\mathbb{P}$ is called \emph{cone homogeneous} if for every pair of conditions $p, q\in \mathbb{P}$ there are extensions $p'\leq p$ and $q'\leq q$ and isomorphism $\phi\colon \mathbb{P}/p'\to \mathbb{P}/q'$.
    \end{definition} 
    \begin{fact}[{\cite[Theorem 26.12]{Jech}}]
        Suppose that $\mathbb{P}$ is a cone homogeneous ordinal definable forcing. Then, $\HOD^{V[G]}\s \HOD^V$ for all generic   $G\s \mathbb{P}$. 
    \end{fact}

 There is a standard forcing that makes any set of ordinals $X$ be part of $\HOD$ of a certain generic extension. This forcing is due to McAloon \cite{McAloon}.  Given a set $X\s \alpha$ and a cardinal $\kappa$, McAloon's forcing  codes $X$ into the power-set-function-pattern  above $\kappa$. The forcing is defined as follows:  Let $\mathcal{E}_X\colon [\kappa^+,\kappa^{+\alpha}]\to \mathrm{Card}$ be the function defined as
$$\mathcal{E}_X(\kappa^{+\beta+1}):=\begin{cases}
    \kappa^{+\beta+2}, & \text{$\beta\in X$};\\
  \kappa^{+\beta+3}, & \text{$\beta\notin X$}.
\end{cases}$$
and denote by $\mathrm{Code}(X,\kappa)$ the Easton support product forcing\footnote{We refer our readers to \cite{Kunen} for the theory of product forcing.} 
$$\prod_{\beta<\alpha}\mathrm{Add}(\kappa^{+\beta+1},\mathcal{E}_X(\kappa^{+\beta+1})).$$
Assume the GCH holds in the interval of cardinals $[\kappa^+,\kappa^{+\alpha}]$. Then, $\mathrm{Code}(X,\kappa)$ codes $X$ into the power-set-function pattern above $\kappa$ in the following sense: If  $G\s\mathrm{Code}(X,\kappa)$ is a generic filter then, in $V[G]$,
$$X=\{\beta<\alpha\mid V[G]\models 2^{\kappa^{+\beta+1}}=\kappa^{+\beta+2}\}.$$
Hence $X$ is ordinal definable in $V[G]$. Since $X\s \mathrm{Ord}$ in fact $X\in \mathrm{HOD}^{V[G]}$.

\begin{fact}
    $\mathrm{Code}(X,\kappa)$ is a $\kappa^+$-directed-closed cone homogeneous forcing.
\end{fact}
Notice that $\mathrm{Code}(X,\kappa)$ is typically not ordinal definable in $V$. Indeed, this is the case in the (interesting) scenario where we want to code some  $X\notin \OD$ into the continuum pattern above $\kappa$. 

\smallskip

By forcing with a class iteration that codes every set of ordinals $X$ into the continuum function through forcings of the form $\mathrm{Code}(X,\kappa)$, for suitably chosen cardinals $\kappa$, McAloon proved the following well-known theorem:

\begin{fact}[McAloon, {\cite{McAloon}}]
  There is a class forcing poset sets yielding a generic extension where the axiom $V=\mathrm{HOD}$ holds.
\end{fact}

A strengthening of $``V=\HOD$" relevant to this paper was introduced by Fuchs--Hamkins-- Reitz in \cite{Geology}:

\begin{definition}[Fuchs--Hamkins--Reitz]
The \emph{generic $\HOD$ of $V$}, denoted $\mathrm{gHOD}$, is the intersection of the $\HOD$s of all the generic extensions of $V$. That is, $x\in \mathrm{gHOD}$ if and only if $x\in \HOD^{V[G]}$ for every set forcing $\mathbb{P}$ and $G\s \mathbb{P}$ generic. Equivalently, 
$$\mathrm{gHOD}=\{x\mid \forall \mathbb{P}\, \one\forces_{\mathbb{P}}\check{x}\in\dot{\HOD}\}$$
\end{definition}

\begin{fact}[{\cite[Theorem 66]{Geology}}]
    McAloon's iteration  forces $V=\mathrm{gHOD}$.
\end{fact}

\medskip

In this manuscript we will be interested in  the possible first places of disagreement (if any) between the set-theoretic universe and $\HOD$. That is, the first ordinal $\alpha$  where $\mathcal{P}(\alpha)\nsubseteq \HOD$. A property that is tightly related to this study is that of \emph{covering}. The first known study of it traces back to Mitchell's theorem on the \emph{Tree Property}  \cite{MitchellAronszajn}. Recently, the property experienced a renewed interest with Hamkins' systematic study in \cite{HamCover}:

\begin{definition}[The cover property]
 Let $M\s N$ be inner models and $A\in M$.    We say that the pair $(M,N)$ has the \emph{full cover property at $A$} if for every uncountable $X\in N$ with $X\s A$  there is $Y\in M$ with $X\s Y$ and $|Y|^N=|X|^N$.
\end{definition}
 
\begin{definition}[$\kappa$-cover property]
  Let $M\s N$ be inner models and let $\kappa$ be a cardinal in $N$. We say that the pair $(M,N)$ has the \emph{$\kappa$-cover property} if given any set $X\in N$ with $X\s M$ and $|X|^N<\kappa$ there is $Y\in M$ with $X\s Y$ and $|Y|^N<\kappa$.

  Given a set $A\in M$, we say that the pair $(M,N)$ has the \emph{$\kappa$-cover property at $A$} if for every $X\in N$ with $X\s A$ and $|X|^N<\kappa$ there is $Y\in M$ with $X\s Y\s a$ and $|Y|^N<\kappa$.
\end{definition}
\begin{remark}
    A case of interest in this paper is $(\HOD, V)$ and $A\in \mathrm{Ord}$.
\end{remark}

The interest in the covering property lies in the fact that it provides a clean combinatorial measure of how closely an inner model approximates the ambient universe. In its standard form, covering asserts that sufficiently small sets in $V$ are contained in comparably small sets belonging to the inner model. Thus, covering gives a precise way of gauging the extent to which $V$ can be approximated from within a canonical inner model.

This idea has played a central role in the inner model program, where covering lemmas are used to analyze how close an $L$-like inner model, or more generally the appropriate core model (if exists)  is to the true universe of sets. For instance, for the relevant core model $\mathcal K$, the non-existence in $V$ of certain large cardinals entails strong forms of covering between $\mathcal K$ and $V$ \cite{MitchellHandbook,SteelHandbook}. A more recent theorem of Woodin shows that, assuming the $\HOD$ Hypothesis, if $\kappa$ is $\HOD$-supercompact, then the pair $(\HOD,V)$ satisfies the $\kappa$-covering property \cite{WoodinPartI}.

\smallskip

The following striking theorem of Magidor, see \cite[Theorem 2.33]{AbrahamMagidor}, illustrates the strength of even a local covering assumption: the $\omega_1$-covering property, together with two relatively mild combinatorial hypotheses, already implies full covering.

\begin{theorem}[Magidor]\label{thm: Magidor}
    Suppose $M$ is an inner model such that:
    \begin{enumerate}
        \item The $\mathrm{GCH}$ holds in $M$.
        \item $M$ is cofinality-correct (i.e., for every ordinal $\alpha$, $\cf^M(\alpha)=\cf(\alpha)$).
        \item $(M,V)$ has the $\omega_1$-covering property. 
    \end{enumerate}
    Then, $(M,V)$ has the full covering property.
\end{theorem}
In his Ph.D. thesis \cite{SebaPhD}, Thei has shown that the GCH can be weakened to the \emph{Singular Strong Hypothesis}. In \S\ref{sec: optimality of Magidors} we shall show that this assumption cannot be removed from the theorem.

\section{The first point of disagreement}\label{sec: compactness HOD}

In this section we are interested in studying the minimal ordinal $\kappa$ --provided  it exists-- for which there is a set $X\s \kappa$ not in $\HOD$. 

\begin{definition}
    Let $(M,N)$ be pair of inner models of $\mathrm{ZFC}$ with $M\subsetneq N$. The \emph{first point of disagreement between $M$ and $N$} is
    $$\Delta(M,N):=\min\{\kappa\in \mathrm{Ord}\mid \mathcal{P}^N(\kappa)\nsubseteq M\}$$
\end{definition}

\begin{proposition}\label{prop: basic coding}
    Let $X\in N$ be a transitive set such that $|X|^N<\Delta(M,N)$. Then, $\mathcal{P}^N(X)\subseteq M$. In particular, $\Delta(M,N)$ is an $N$-cardinal.
\end{proposition}
\begin{proof}
    Denote  $\kappa:=|X|^N$ and let $\pi\colon \kappa \to X$ be a bijection in $N$. Consider $R\s \kappa\times \kappa$ the relation defined by $$\text{$\beta\, R\, \gamma$ if and only if $\pi(\beta)\in \pi(\gamma)$.}$$ Clearly, $\pi\colon (\kappa, R)\to (X,\in)$ is an isomorphism of structures. Since $(X,\in)$ is transitive, $\pi$ must be the Mostowski collapse map on $(\kappa,R)$. 
    Let $R^* \s \kappa$ be the result of applying Gödel's pairing function to the relation $R$. Since we assume that  $\kappa < \Delta(M,N)$ it follows that  $R^*\in M$ and by absoluteness of Gödel's pairing and of the collapse we deduce that  $\pi\in M$. 

Let $A\in \mathcal{P}^N(X)$, and set $B:=\pi^{-1}[A]$. Use our assumption to infer that $B\in M$. Since $\pi\in M$ we conclude that $A\in M$, concluding that $\mathcal{P}^N(X)\subseteq M$.

 To see the second part we argue as follows:  If  $\Delta(M,N)$ were not an $N$-cardinal then $|\Delta(M,N)|^N<\Delta(M,N)$ and by the first part, $\mathcal{P}(\Delta(M,N))\s M$. However, by the very definition of $\Delta(M,N)$, $\mathcal{P}^N(\Delta(M,N))\not\subseteq M$. A contradiction.
\end{proof}

The case of primary interest for us is when the ambient universe $N$ coincides with the set-theoretic universe, $V$, and the inner model $M$ equals $\HOD$.  In the forthcoming sections we discuss what are the possible first points of disagreements between $V$ and $\HOD$. We divide the discussion into three families of cardinals: regulars, singulars of countable cofinality and singulars of uncountable cofinality.

\subsection{Regular cardinals}
In the context of regular cardinals, any of them can be the first point of disagreement between $V$ and $\HOD$.\footnote{Just start from a model of $V=g\HOD$ and force with $\Add(\kappa,1)$.} 
\label{Fact: critical point can be measurable}
In fact if $\kappa$ possess any mild large cardinal property (say, Inaccessibility, Mahloness, Weak Compactness, etc.) then $\kappa$ can be the first point of disagreement and still retain this large property in the obtained extension. The same conclusion applies if $\kappa$ is a   supercompact cardinal, by dint of Laver's  \cite{Lav}.

As it turns, there is a large-cardinal threshold for which such a flexibility disappears. This is the case of the \emph{$\HOD$-large cardinal hierarchy}, which stems from work of Woodin on the \emph{HOD Conjecture} \cite{WoodinPartI}. HOD-supercompactness is a pivotal notion for the study of HOD, being extensively studied by Woodin \cite{midrasha, WoodinPartI}, Sargsyan \cite{Sargsyan2008}, Apter--Friedman--Fuchs \cite{APTER2021102901} and Goldberg--Osinski--Poveda \cite{GOP}, among others.

\smallskip

The following natural variation of Woodin's notion was studied in  \cite{APTER2021102901}:

\begin{definition}[Apter-Friedman-Fuchs {\cite{APTER2021102901}}]
    Let $N$ be an inner model. Let $\kappa$ be a cardinal and $X$ a set. Then $\kappa$ is $(N, X)$-measurable if
there is a $j \colon V \to M$ with $\crit(j)=\kappa$, $j(\kappa) > \text{rnk}(X)$ and
$j(N) \cap X = N \cap X$.
\end{definition}
As the authors observed, an equivalent formulation for $(N,\mathcal{P}(\kappa))$-measura\-bility is the existence of a normal ultrafilter $U$ on $\kappa$, such that $$\text{$\langle N,U\cap N,\in\rangle\models ``U\cap N$ is a normal measure on $\kappa$"}$$ which is weakly amenable\footnote{We say that $U$ is weakly amenable to $N$ if for any sequence $\langle X_\xi\mid \xi<\kappa\rangle\subseteq \mathcal{P}(\kappa)\cap N$ in $N$, we have that $\{\xi<\kappa\mid X_\xi\in U\}\in N$.} to $N$.
\begin{proposition}\label{Prop: first point is not HOD-measurable}
    If $\kappa$ is $(\HOD,\mathcal{P}(\kappa))$-measurable, then $\Delta(\HOD, V)\neq \kappa$.
\end{proposition}
\begin{proof}
    Suppose otherwise and let $j:V\to M$ be an embedding witnessing $(\HOD,\mathcal{P}(\kappa))$-measurability. Let $Y\in \mathcal{P}(\kappa)\setminus \mathcal{P}^{\HOD}(\kappa)$. Since $\crit(j)=\kappa$, $Y\in M\cap \mathcal{P}(\kappa)$. Note that $Y\notin \mathcal{P}(\kappa)\cap \HOD^M$, since by assumption $\HOD\cap \mathcal{P}(\kappa)=\HOD^M\cap \mathcal{P}(\kappa)$. Hence $M$ thinks that  $j(\kappa)>\crit^{\HOD^M}(M)$, contradiction.  
\end{proof}
Apter--Friedman--Fuchs proved that any cardinal $\kappa$ which is $(\HOD,\mathcal{P}(\kappa))$-measurable, must be ineffable, and in fact totally indescribable (i.e. $\Pi^1_n$-indescribable for every $n<\omega$) in $\HOD$. They also show the consistency of being measurable (even supercompact) but not $(\HOD,\mathcal{P}(\kappa))$-measurable. Here we see that it straightforwardly follows from Proposition \ref{Prop: first point is not HOD-measurable}, indeed, we already observed that we can make a measurable (even a supercompact) cardinal be the the first point of disagreement point, which then cannot be $(\HOD,\mathcal{P}(\kappa))$-measurable.

\smallskip

The authors ask (see Question 2.9 in \cite{APTER2021102901}) whether there is a larger cardinal notion that can be derived from $(N,\mathcal{P}(\kappa))$-measurability. Let us show that we can for $N=\HOD$ (which is the main focus of \cite{APTER2021102901}): 
\begin{proposition}
    Suppose that $\kappa$ is $(\HOD,\mathcal{P}(\kappa))$-measurable, then $\kappa$ is completely ineffable\footnote{We will use the equivalent definition as presented in Benhamou-Gitman \cite[Cor. 5.7]{Benhamou_Gitman}: namely $\kappa$ is completely ineffable, if in a forcing extension there is a weakly amenable $V$-normal ultrafilter on $\kappa$.} in $\HOD$.
\end{proposition}
Before proving the proposition, let us mention that it is well known that completely ineffable cardinals are totally indescribable (see \cite{AbramsonHarringtonKleinbergZwicker:FlippingProperties} for more on completely ineffable cardinals).

\begin{proof}
   Assume that $\kappa$ is $(\HOD,\mathcal{P}(\kappa))$-measurable and let $U$ be an ultrafilter over $\kappa$, such that $U\cap\HOD$ is a normal weakly amenable $\HOD$-ultrafilter. By Vopenka's thoerem (see e.g. \cite[Theorem 6]{WoodinDavisRodriguez}) $U\in \HOD[G]$ for some generic filter $G$ over $\HOD$. 
We conclude that in a generic extension of $\HOD$, there is a $\HOD$-ultrafilter, $\HOD$-normal which is weakly amenable ultrafilter. By Benhamou-Gitman \cite[Cor. 5.7]{Benhamou_Gitman}, the existence of such an ultrafilter in a generic extension of $\HOD$ implies that $\kappa$ is completely ineffable in $\HOD$. 
\end{proof}
In fact we can deduce more--  $U$ is a normal measure in $V$ and therefore can be iterated. This is the work of Gitman \cite{gitman:ramsey} on the $\alpha$-iterability property.

\subsection{Singular cardinals of countable cofinality}
In this section we show that 
$\aleph_\omega$ can be the first point of disagreement between $\HOD$ and $V$. Our theorem stems from a question posed by Hayut in \href{https://mathoverflow.net/questions/201555/can-the-first-ordinal-in-which-v-neq-hod-be-aleph-omega}{MathOveflow} in 2015: 
\begin{question}[Hayut]
  \emph{Can $\aleph_\omega$ be the first cardinal where $V\neq \HOD$?}  
\end{question}

An affirmative answer was provided by M. Golshani on the very same post assuming the existence of a cardinal $\kappa$ that is $(\kappa+2)$-strong. Golshani conjectured that the same configuration could be obtained from the optimal assumptions. Here we prove Golshani's conjecture:
\begin{theorem}\label{thm: alephomega}
    The following theories are equiconsistent, modulo $\mathrm{ZFC}$:
    \begin{enumerate}
        \item ``There is an inner model with a  measurable cardinal".
        \item ``$\Delta(\HOD, V)=\aleph_{\omega}$".
    \end{enumerate}
\end{theorem}
The implication (2) $\Rightarrow$ (1) is an easy application of the Dodd--Jensen covering lemma:
\begin{lemma}
 If  $\Delta(\HOD, V)=\aleph_{\omega}$ then there is an inner model with a measurable cardinal.  
\end{lemma}
\begin{proof}
    Suppose for the sake of a contradiction that  there is no inner model with a measurable cardinal yet $\aleph_\omega$ is the first cardinal $\lambda$ with $\mathcal{P}(\lambda)\nsubseteq \HOD$.

    Let $X\in \mathcal{P}(\aleph_\omega)$ be an offending set. For each $n<\omega$, $X\cap \aleph_n \in \HOD$. 

    For each $n<\omega$ let $\rho_n$ be the rank of $X\cap \aleph_n$ in the canonical well-order of $\HOD$. Consider $R= \{\rho_n\mid n<\omega\}$. By the Dodd--Jensen Covering Theorem (cf. \cite{MitchellCoveringHandbook}) there is $S\in \mathcal{K}^{\mathrm{DJ}}\s \HOD$ such that $R\s S$ with $|S|\leq \aleph_1$. 

    Let $\alpha:=\mathrm{otp}(S)$ and $\pi \colon S \to \alpha$ be the unique isomorphism; namely, $\pi$ is just the Mostowski collapse map. Since $S\in \HOD$, by absoluteness of the Mostowski collapse,
     $\pi\in \HOD$. Finally, $\pi[R]\in \HOD$ --in that a subset of $\alpha < \aleph_\omega=\Delta(\HOD, V)$-- so that $R\in \HOD$ as well.

From $R$ we can uniformly define $X$ using ordinal parameters.
\end{proof}

Let us next show $(1)\Rightarrow (2)$. Our  proof will feature  ideas similar to those used by Magidor in his classical paper \cite{MagOnTheSingI}, or those used by Woodin in his (unpublished) construction of an $\aleph_\omega$-indecomposable ultrafilter.

\smallskip

 Assume that $\kappa$ is a measurable cardinal such that $2^\kappa=\kappa^+$. Let $\mathcal{U}$ be a $\kappa$-complete normal ultrafilter and $j\colon V\to M$  the ultrapower elementary embedding by $\mathcal{U}$. Standard arguments (see Cummings \cite[\S8]{CummingsHandbook}) give a $M$-generic filter $K\in V$  for the poset $\big(\mathrm{Col}((\kappa^{++})^{M},{<}j(\kappa))\big)^{M}$.

\begin{notation}
   Given a set $X\s \mathrm{Ord}$ and $n<\omega$, 
   $$X^n:=\{(\alpha_0,\dots, \alpha_{n-1})\in \underbrace{X\times\cdots \times X}_{n\, \text{times}}\mid \alpha_0<\dots<\alpha_{n-1}\}.$$
\end{notation}

\smallskip

Our variant of Prikry forcing with collapses is defined as follows: 
\begin{definition}\label{def: Prikryforcing}
    Let $\mathbb{P}:=\mathbb{P}(\mathcal{U},K)$ be the poset whose conditions are
    $$p=(\vec{\kappa}_p,\vec{c}_p, A_p, h_p, H_p)$$
   for which the following properties hold:
    \begin{enumerate}
        \item $\vec{\kappa}_p=\langle \kappa_i\mid i< \ell(p)\rangle$ is a finite increasing sequence  of inaccessible cardinals less than $\kappa$ and $\ell(p)<\omega$ is an odd number.
        \item $\vec{c}_p=\langle c_i\mid -1\leq i\leq m(p)\rangle$ is a finite sequence of functions such that: 
        \begin{enumerate}
        \item $m(p):=\mathrm{floor}\left(\frac{\ell(p)-1}{2}\right)-1$.\footnote{Here $\mathrm{floor}(x)$ denotes the function giving the integer part of $x$.}
            \item  $c_{-1}\in \mathrm{Col}(\omega_2,{<}\kappa_0)$;
            \item  $c_i\in \mathrm{Col}(\kappa_{2  i}^{++}, {<}\kappa_{2  (i+1)})$ for $i \leq  m(p)$,
        \end{enumerate}

\item $A_p\in \mathcal{U}$ consists of inaccessible cardinals and   $$\min(A_p)>\max\{\kappa^p_{\ell(p)-1},\sup(\mathrm{range}(h_p))\}.$$
\item $h_p\in  \mathrm{Col}(\kappa_{\ell(p)-1}^{++},{<}\kappa).$
\item $H^p$ is a function with domain $A_p$ such that $$H_p(\beta)\in \mathrm{Col}(\beta^{++},{<}\kappa)$$ and $j(H)(\kappa)\in K$.
    \end{enumerate}
Given $p\in \mathbb{P}$ we shall refer to  $(\vec{\kappa}_p,\vec{c}_p)$ as the \emph{stem of $p$} and to $\ell(p)$  as the \emph{length of $p$}. When the context allows it, we will tend to omit the sub and superscripts on the components of the condition $p$.

\smallskip

    For conditions $p,q\in \mathbb{P}$ we will write $q\leq^* p$ whenever the following hold:
    \begin{enumerate}
        \item $\ell(p)= \ell(q)$;
        \item $\langle \kappa^q_i\mid i<\ell(p)\rangle = \langle \kappa^p_i\mid i<\ell(p)\rangle$;
        \item For each $-1\leq i\leq m(p)$, $c^p_i\s c^q_i$;
        \item $A_q\s A_p$;
        \item $h^p\s h^q$;
        
        \item $H^p(\beta)\s H^q(\beta)$ for all $\beta\in A_q$.
    \end{enumerate}
\end{definition}

\begin{definition}
    Given $p\in \mathbb{P}$  and $(\rho,\sigma)\in A^2$. Define
    $$p\cat (\rho, \sigma):=(\vec{\kappa}\,{}^\smallfrown \langle \rho,\sigma\rangle, \vec{c}\,{}^\smallfrown\langle h\rangle, A^*, H(\sigma), H^*)$$
     where we have set 
 \begin{itemize}
     \item $A^*:=A\setminus \max\{\sigma+1,\sup(\mathrm{range}(H(\sigma))\},$
     \item and $H^*:=H\restriction A^*$.
 \end{itemize}
 Given a finite sequence $\langle (\rho_i,\sigma_i)\mid i<n\rangle$ of members of $A^2$ such that  $\sigma_i<\rho_{i+1}$ for all $i+1<n$, we define $p\cat \langle (\rho_i,\sigma_i)\mid i<n\rangle$ by recursion of $i<n$.
 \end{definition}

 The following proposition is routine:
 \begin{proposition}
    For each $p\in \mathbb{P}$ and a sequence $\langle (\rho_i,\sigma_i)\mid i<n\rangle$  as before, the vector $p\cat \langle (\rho_i,\sigma_i)\mid i<n\rangle$ is a condition in $\mathbb{P}$.\qed
 \end{proposition}
 \begin{definition}
     Given conditions $p,q\in \mathbb{P}$ we shall write $q\leq p$ if there is a finite sequence $\langle (\rho_i,\sigma_i)\mid i<n\rangle$ of members of $(A_p)^2$ with  $\sigma_i<\rho_{i+1}$ for all $i+1<n$, such that $q\leq^* p\cat \langle (\rho_i,\sigma_i)\mid i<n\rangle$.
 \end{definition}
  Clause~(5) in Definition~\ref{def: Prikryforcing} ensures that any two conditions $p,q\in \mathbb{P}$ with the same stem are $\leq^*$-compatible. As a result we have the following:
\begin{lemma}
   $\mathbb{P}$ has the $\kappa^{+}$-cc.\qed
\end{lemma}

To control the cardinal structure below $\kappa$, one needs the \emph{Prikry property}, whose proof is standard, and we therefore omit it. The reader interested in the details may consult \cite[Lemma~4.11]{NonGalvinFil}, where a very similar argument is given in full detail:
\begin{lemma}
Given  $p\in \mathbb{P}$ and $\varphi$ a sentence in the language of forcing of $\mathbb{P}$ there is $q\leq^* p$ such that either $q\forces_{\mathbb{P}}\varphi$ or $q\forces_{\mathbb{P}}\neg \varphi$.\qed
\end{lemma}

 The poset  $\mathbb{P}$ possess the following factoring property that ensures that every bounded subset of $\kappa$ belongs to a generic extension of by the interwoven Levy collapses. More precisely, for each  $p\in \mathbb{P}$, 
$$ \mathbb{P}/p\simeq \mathrm{Coll}(\omega_2,{<}\kappa_0)/c^p_{-1}\times \left(\prod_{i\leq m(p)}(\mathrm{Coll}((\kappa^p_{2i})^{++},{<}\kappa^p_{2(i+1)})/c^p_i)\right)\times \mathbb{P}(p),$$
where $\langle \mathbb{P}(p),\leq, \leq^*\rangle$ is a Prikry-type subposet of $\langle \mathbb{P},\leq, \leq^*\rangle$ whose  $\leq^*$-order is $\kappa_{2  m(p)}^{++}$-closed. For details about how this  property helps to show that every bounded subset of $\kappa$ in the generic extension is added by a product of this Levy collapses see for instance \cite[Lemma~3.15]{PartIII}.

\smallskip

Let $G\s \mathbb{P}$ be a  generic filter. As usual, $G$ induces a  sequence $\langle \kappa_n\mid n<\omega\rangle$ cofinal in $\kappa$, together with a  sequence $\langle C_n\mid n\geq -1\rangle $ of generic filters $C_{-1}\s \mathrm{Coll}(\omega_2,{<}\kappa_0)$ and $C_n\s \mathrm{Coll}(\kappa_{2n}^{++}, {<}\kappa_{2(n+1)})$, respectively. Combining the previous factoring, the Prikry property  and the closure of the $\leq^*$-ordering one can show that the only $V[G]$-cardinals below $\kappa$ are $$\textstyle \{\aleph_0,\aleph_1^V, \aleph_2^V\}\cup \{\kappa\}\cup \bigcup_{n<\omega}\{\kappa_{2n},\kappa_{2n}^+,\kappa^{++}_{2n}\}.$$
In particular:
\begin{lemma}
   The trivial condition of $\mathbb{P}$ forces $``\kappa=\aleph_\omega$".\qed 
\end{lemma}

We now turn our attention to $\HOD$ of the generic extension by $\mathbb{P}$.

\smallskip

Consider $\mathbb{Q}$ the poset whose universe is
$$\{(\mathrm{Even}(\vec{\kappa}_p),\vec{c}_p, A_p, h_p, H_p)\mid p\in \mathbb{P}\},$$
where $\mathrm{Even}(\vec\kappa_p)$ is the sequence of even coordinates from $\vec{\kappa}_p = \langle \kappa_i^p\mid i<\ell(p)\rangle$.

Let $\pi\colon \mathbb{P}\to \mathbb{Q}$ be the forgetful map that sends any $p\in \mathbb{P}$ to the corresponding vector in $\mathbb{Q}$. Given $p\in \mathbb{P}$ and $(\rho,\sigma)\in A^p$  define 
$$\pi(p)\cat \langle (\rho,\sigma)\rangle:=\pi(p\cat \langle (\rho,\sigma)\rangle)$$
and $\pi(p)\cat \langle (\rho_i,\sigma_i)\mid i<n\rangle$ by recursion in the obvious fashion.  Define $\pi(q)\leq^* \pi(p)$ in the natural way (i.e., as in Definition~\ref{def: Prikryforcing})  and stipulate that  $\pi(q)\leq \pi(p)$ whenever $\pi(q)\leq^*\pi(p)\cat \langle (\rho_i,\sigma_i)\mid i<n\rangle$.

\smallskip

 As it turns out, $\langle \mathbb{Q},\leq,\leq^*\rangle$ is  a Prikry-type forcing with similar properties as those of  $\mathbb{P}$. Also, $\pi\colon \mathbb{P}\to \mathbb{Q}$ is a projection (both with respect to $\leq$ and $\leq^*$).  So, in particular, given a generic $G\s \mathbb{P}$,  $\pi[G]$ is generic for  $\mathbb{Q}$ and our earlier comments imply that $V[\pi[G]]$ and $V[G]$ posses the same bounded subsets of $\kappa$ (i.e., of $\aleph_\omega^{V[G]}$).
\begin{lemma}[Homogeneity]
\label{lemma: homogeneity}
  Let $p, q\in \mathbb{P}$ be conditions such that $\pi(p)$ and $\pi(q)$ are $\leq^*$-compatible. Then, there are $p^*\leq^* p$ and $q^*\leq^*q$ together with an isomorphism $$\Gamma\colon \mathbb{P}/p^*\rightarrow \mathbb{P}/q^*.$$
  Moreover, for any generic filter $G\s \mathbb{P}/p^*$, the collapsing generics induced from $G$ are the same as those induced from $\Gamma[G]$; namely, $\pi[G]=\pi[\Gamma[G]]$.\footnote{Here we are identifying $\Gamma[G]$ with the (generic) filter that it induces. } 
\end{lemma}
\begin{proof}
Let $p= (\vec{\kappa}_p,\vec{c}_p, A_p, h_p, H_p)$ and $q=(\vec{\kappa}_q,\vec{c}_q, A_q, h_q, H_q)$ be as before. For each $i<m$ let $c^*_i:=c^p_i\cup c^p_i$. Since $j(H_p)(\kappa),j(H_q)(\kappa)\in K$ there is $B\s A_p\cap A_q$ in $\mathcal{U}$ such that $H_p(\beta)$ is compatible with $H_q(\beta)$. Set $A^*:=B$, $h^*=h_p\cup h_q$ and $H^*$ be the function with domain $A^*$ such that $H^*(\beta):=H_p(\beta)\cup H_q(\beta)$. Let $p^*$ be the vector obtained after replacing $\vec{c}_p$, $A_p$, $h_p$ and $H_p$ by $\vec{c}^*$, $A^*$, $h^*$ and $H^*$, respectively. Define $q^*$ similarly.

Clearly, $p^*\leq^* p$ and $q^*\leq^* q$.

There is a natural isomorphism $\Gamma\colon \mathbb{P}/p^*\to \mathbb{P}/q^*$ that given $r\leq p^*$ sends it to $\Gamma(r)$, the vector swapping the odd Prikry points of $p^*$, $\langle \kappa^{p}_{2(i+1)}\mid i\leq m(p)\rangle$, to the odd Prikry points of $q^*$, $\langle \kappa^q_{2(i+1)}\mid i\leq m(p)\rangle$, keeping the other entries unaltered. By construction,  $\pi[\Gamma[G]]=\pi[G]$ for all $G\s \mathbb{P}/p^*$ generic.
\end{proof}

\begin{lemma}
   $\HOD^{V[G]}_x\s V[\pi[G]]$ for every set $x\in V[\pi[G]]$.
\end{lemma}
\begin{proof}
 By Proposition \ref{prop: basic coding}, it suffices to show that every set of ordinals in $\HOD^{V[G]}_x$ is in $V[\pi[G]]$. Let  $a\in \HOD^{V[G]}_x$ be a subset of some $\chi\in \mathrm{Ord}$. By definition, $$a=\{\beta<\chi \mid V[G]\models \varphi(\beta, \alpha_0,\dots, \alpha_{k-1},x)\}$$ for ordinals $\alpha_0,\dots,\alpha_{k-1}$ and a formula $\varphi(y, y_0,\dots, y_{k-1},z)$ in $\mathcal{L}_\in$.

 \smallskip

 {Let $\dot{x}$ be a $\mathbb{Q}$-name such that $\dot{x}_G=x$ that is invariant under any isomorphism $\Gamma$ such that $\pi[\Gamma[G]]=\pi[G]$}. Then, we have
 $$a=\{\beta<\chi\mid \exists p\in G\, (p\forces_{\mathbb{P}}\varphi(\check{\beta}, \check{\alpha}_0,\dots,\check{\alpha}_{k-1},\dot{x}))\}.$$

 \smallskip
 
 Define $$\bar{a}=\{\beta<\chi\mid \exists p\in \mathbb{P}/\pi[G]\, (p\forces_{\mathbb{P}}\varphi(\check{\beta}, \check{\alpha}_0,\dots,\check{\alpha}_{k-1},\dot{x})\}.$$ Clearly, $\bar{a}\in V[\pi[G]]$ and $a\s \bar{a}$. We claim that $\bar{a}\s a.$

 \smallskip

 Suppose that $\beta\in \bar{a}\setminus a$. Then, there is $p\in \mathbb{P}/\pi[G]$ and $q\in G$ such that 
 $$p\forces_{\mathbb{P}}\varphi(\check{\beta}, \check{\alpha}_0,\dots,\check{\alpha}_{k-1},\dot{x})\,\text{and}\,
 q\forces_{\mathbb{P}}\neg \varphi(\check{\beta}, \check{\alpha}_0,\dots,\check{\alpha}_{k-1},\dot{x}).$$
 By extending both $p$ and $q$ if necessary we may ensure that $p\in G$ and $q\in \mathbb{P}/\pi[G]$ and that both have the same length. Note that in these circumstances we can invoke the previous lemma and find conditions $p^*\leq^* p$, $q^*\leq^* q$ together with an autormorphism $\Gamma$ sending $p^*$ to $q^*$, such that $\pi[G]=\pi[\Gamma[G]]$. Since $\Gamma$ does not change the  names involved in the  formula we get  $\beta \in a\cap \bar a$. Contradiction.
\end{proof}

\begin{theorem}\label{thm: Deltaequalsalephomega}
    Assume that there is a measurable cardinal. Then, there is a generic extension where $\aleph_\omega$ is strong limit  and $\Delta(\HOD, V)=\aleph_\omega$.
\end{theorem}
\begin{proof}
Let $\kappa$ be a measurable cardinal.
   Without loss of generality, we may assume that our ground model satisfies $``\kappa$ is measurable'', $V=\mathrm{gHOD}$ and  GCH holds on a large enough interval of cardinals. Let  $\mathbb{P}$ be the previous poset and  $G\s \mathbb{P}$ a generic filter. Let $x$ be a set of ordinals such that  $V[x]=V[\pi[G]]$. Let $\mathrm{Code}(x,\aleph_\omega)$ be the poset coding $x$ into the power-set-function pattern above $\aleph_\omega$. Let $c \s \mathrm{Code}(x,\aleph_\omega)$ be generic over $V[G]$. Since $\mathrm{Code}(x,\aleph_\omega)$ is weakly homogeneous and $\mathrm{OD}^{V[G]}_x$,
   $$V[\pi[G]]=V[x]\s \HOD^{V[G][c]} \s \HOD^{V[G]}_x\s V[\pi[G]],$$
  where the first inclusion follows from $V=\mathrm{gHOD}$ and the last inclusion follows from the previous lemma.  Thus, $\HOD^{V[G][c]}=V[\pi[G]]$.

  \smallskip
  
  In particular, $\HOD^{V[G][c]}$ have the same bounded subsets of $\aleph_\omega$ as $V[G]$ in that this is the case between $V[G]$ and $V[\pi[G]]$. So, to prove that $$\Delta(\HOD^{V[G][c]}, V[G][c])=\aleph_\omega,$$ it suffices to note that $\{\kappa_{2n+1}\mid n<\omega\}$ is not in $V[\pi[G]]$.

  \smallskip
  
    Suppose otherwise -- that $\{\kappa_{2n+1}\mid n<\omega\}\in V[\pi[G]]$. Let $\tau$ be a $\mathbb{Q}$-name such that $\tau_{\pi[G]}=\{\kappa_{2n+1}\mid n<\omega\}$ and $p\in G$  forcing this.

    Since $A':=\{\alpha\in A_p\mid \text{$A_p\cap \alpha$ is unbounded in $\alpha$}\}$ is measure one, the Mathias criterion of genericity gives  $n_*<\omega$ such that $\{\kappa_{n}\mid n\geq n^*\}\s A'$. To simplify notations, let us assume that $n_*=\ell(p)+1$.\footnote{Otherwise we just extend $p$ with the Prikry points $\{\kappa_n\mid \ell(p)\leq n\leq n_*-2\}$.}

  Let $p_0\leq p\cat  (\kappa_{\ell(p)}, \kappa_{\ell(p)+1})$ be in $G$. Given $\alpha\in A^p\cap (\kappa_{\ell(p)-1},\kappa_{\ell(p)+1})$ distinct from $\kappa_{\ell(p)}$, let $p_1$ be a condition such that  $p_1\leq p\cat  (\alpha, \kappa_{\ell(p)+1})$, $\ell(p_0)=\ell(p_1)$, $\vec{c}_{p_1}=\vec{c}_{p_0}$, $A_{p_1}=A_{p_0}$, $h_{p_1}=h_{p_0}$, $H_{p_1}=H_{p_0}$. 
  
  Note that such a condition $p_1$ always exists in that adding the odd-points (i.e., $\kappa_{\ell(p)}$ and $\alpha$) to $p$ does not have any effect on the values of $H_p$. 

    Applying Lemma~\ref{lemma: homogeneity} we obtain an isomorphism $\Gamma\colon \mathbb{P}/p_0\to \mathbb{P}/p_1$ such that $\pi[{\Gamma[G]}]=\pi[G]$. Since $p_0,p_1\leq p$ it follows that $p_i\forces_{\mathbb{P}} ``\tau=\dot{\vec{\kappa}}_{\mathrm{odd}}$" where $\dot{\vec{\kappa}}_{\mathrm{odd}}$ is the standard $\mathbb{P}$-name for the odd members of the Prikry sequence. As such, $p_0\forces_{\mathbb{P}}``\check{\kappa}_{\ell(p)}\in {\tau}$" and $p_1\forces_{\mathbb{P}}``\check{\alpha}\in {\tau}$". Since $p_0\in G$, $p_1\in \Gamma[G]$, ergo, $\kappa_{\ell(p)}\in \tau_{\pi[G]}=\{\kappa_{2n+1}\mid n<\omega\}$ and $\alpha\in \tau_{\pi[\Gamma[G]]}=\tau_{\pi[G]}$. Therefore,
    $$\kappa_{\ell(p)},\alpha\in\{\kappa_{2n+1}\mid n<\omega\}$$
    but this is impossible because $\alpha\in (\kappa_{\ell(p)-1},\kappa_{\ell(p)+1})$ and $\alpha\neq \kappa_{\ell(p)}$.
\end{proof}

\subsection{Singular cardinals of uncountable cofinality}\label{sec: compactness HOD singulars}

We begin this section with the following spin-off of a theorem of Goldberg--Poveda \cite{GolPov}:
\begin{theorem}\label{thm: spinoffGP}
 Suppose that $\kappa$ is a singular cardinal of uncountable cofinality such that $\cf(\lambda)=\cf^{\HOD}(\lambda)$ for all $\lambda < \kappa$. Then, $\cf(\kappa)=\cf^{\HOD}(\kappa)$.   
\end{theorem}
\begin{proof}
It suffices to prove that $\cf^{\HOD}(\kappa)\leq \cf(\kappa)$. So we  suppose that this is not the case and let us produce a contradiction from this assumption.

Since $\kappa\cap \cof(\omega)$ is a  stationary consisting of $\HOD$-singular cardinals, \cite[Theorem~3.5]{GolPov} ensures that  $\cf^{\HOD}(\kappa)<\kappa$. Set $\delta :=\cf^{\HOD}(\kappa)$. 

To complete the proof it would suffice to check that $\cf(\delta)=\cf(\kappa)$ in that
$$\cf(\kappa)=\cf(\delta)=\cf^{\HOD}(\delta)=\delta=\cf^{\HOD}(\kappa).$$
Let $f\colon \delta\to \kappa$ be a cofinal, increasing, continious function in $\HOD$.

The existence of this function outright ensures that $\cf(\kappa)\leq \cf(\delta)$. For the converse inequality, let $c\s \kappa$ be a club set (in $V$) with $\mathrm{otp}(c)=\cf(\kappa)$. Define $h\colon \kappa\to \delta$ by $h(\eta):=\min\{\alpha <\delta \mid f(\alpha)\geq \eta\}.$ Clearly, $h[c]$ is a cofinal subset of $\delta$ with cofinality $\cf(\kappa)$, ergo $\cf(\delta)\leq \cf(\kappa)$. 
\end{proof}

\begin{theorem}\label{thm: compactness}
    Suppose that   $\kappa$ is a singular cardinal  of uncountable cofinality such that   $2^{\cf(\kappa)}<\kappa$. Then,   $\Delta(\HOD,V)\neq \kappa$.
\end{theorem}
\begin{proof}
   Suppose towards a contradiction that  $\Delta(\HOD, V)=\kappa$. Our proof will combine Theorem~\ref{thm: spinoffGP} with ideas from a theorem of Shelah \cite{ShelahHOD}. Note that the assumptions of Theorem~\ref{thm: spinoffGP} are met, hence $$\cf^{\HOD}(\kappa)=\cf(\kappa).$$

   \begin{claim}
       $H_\kappa=H_{\kappa}^{\HOD}$.
   \end{claim}   \begin{proof}[Proof of claim]
      It remains to prove the left-to-right inclusion. Let $x\in H_\kappa$ then $x\in P(\mathrm{tcl}(\{x\}))$. Since $|\mathrm{tcl}(\{x\})|<\kappa$, we may apply Proposition \ref{prop: basic coding} to conclude that $x\in H_\kappa^{\HOD}$.
   \end{proof}

  Working in $\HOD$, let $\vec{\delta}=\langle \delta_i\mid i<\cf(\kappa)\rangle$ and $\vec{t}=\langle t_i\mid i<\theta \rangle$ be, respectively, a cofinal, increasing, continuous sequence in $\kappa$ and an injective enumeration of $H_\kappa$ (i.e., of $H_\kappa^{\HOD}$). From this point onward, we argue similarly to a  theorem of  Shelah saying that there is an $\mathrm{OD}_x$ well-ordering of $\mathcal{P}(\kappa)$ for some $x\s \kappa$ \cite{ShelahHOD}. As we will see, any set $x\s \kappa$ such that $\vec{\delta},\vec{t}\in \HOD_{\{x\}}$ works to grant this inclusion. In our particular case,  both $\vec{\delta},\vec{t}$ are already in $\HOD$ so $x=\emptyset$ will work, and as a result $\mathcal{P}(\kappa)\s \HOD$.

   \smallskip

   For each $X\s \kappa$ we define $f_X\colon \cf(\kappa)\to \theta$  by $$\text{$f_X(i):=``$The unique $j<\theta$ such that $X\cap \delta_i=t_j$".}$$
   Define the following relation on $\mathcal{P}(\kappa)$:
   $$X\lhd Y\;\Longleftrightarrow\; \{i<\cf(\kappa)\mid f_X(i)\geq f_Y(i)\}\;\text{is bounded in $\cf(\kappa)$}.$$

    Note that if $X\neq Y$ then $f_X(i)\neq f_Y(i)$ for a tail of $i<\cf(\kappa)$. Also, notice that   $\lhd$ is well-founded: Suppose that $\langle X_n\mid n<\omega\rangle$ were to be a $\lhd$-decreasing sequence of subsets of $\kappa$. Since $\cf(\kappa)\geq \omega_1$ there is $i<\cf(\kappa)$ such that $f_{X_{n+1}}(i)<f_{X_n}(i)$ for all $n<\omega$, which would contradict the well-foundedness of the ordinals.

  \smallskip

   For each $\rho\in \mathrm{Ord}$ let $R_\rho:= \{X\s \kappa\mid \mathrm{rank}_{\lhd}(X)=\rho\}$.

   \begin{claim}
       For each $\rho\in \mathrm{Ord}$, $|R_\rho|\leq 2^{\cf(\kappa)}<\kappa$.
   \end{claim}
   \begin{proof}[Proof of claim]
   This follows from the Erdös--Rado theorem.

   Towards a contradiction, let $\{X_\alpha\mid \alpha <(2^{\cf(\kappa)})^+\}$  be an injective enumeration of pairwise distinct members of $R_\rho$. Define a coloring $$c\colon [(2^{\cf(\kappa)})^+]^2 \to \cf(\kappa)$$ by stipulating that $$c(\alpha,\beta):=\min\{i<\cf(\kappa)\mid f_{X_\alpha}(i)>f_{X_\beta}(i)\}.$$
  This is well-defined because given $X, Y\in R_\rho$ and $X\neq Y$ there are infinitely many $i<\cf(\kappa)$ such that $f_X(i)>f_Y(i)$. Thus, by the Erdös--Rado theorem, there is $H\s (2^{\cf(\kappa)})^+$ with $|H|=\cf(\kappa)^+$ that is monochromatic -- say, with color $i^*$. This is impossible, for it will imply that $f_{X_\alpha}(i^*)>f_{X_\beta}(i^*)$ for all $\alpha<\beta$ in $H$, thus contradicting well-foundedness of the ordinals.
   \end{proof}
For each ordinal $\rho$, let $A_\rho :=\bigcup\{\mathrm{ran}(f_X)\mid X\in R_\rho\}$. By the above observation $|A_\rho|<\kappa$. Let $g_\rho\colon A_\rho \to \mathrm{otp}(A_\rho)$ be the unique order-preserving bijection and, for each $X\in R_\rho$,  $g_\rho \circ f_X\colon \cf^{\HOD}(\kappa) \to \mathrm{otp}(A_\rho)$. Since this function belongs to $H_\kappa$ it also belongs to $H_\kappa^{\HOD}$. In particular, there is a unique index  $i_X<\theta$ in the enumeration of $\vec{t}$ such that $g_\rho\circ f_X= t_{i_X}$.

\smallskip

Define a well-order $\prec_\rho$ over each $R_\rho$ by saying that $X\prec_\rho Y$ if and only if $i_X<i_Y$, and a well-order $\prec$ of $\mathcal{P}(\kappa)$ by stipulating that $X\prec Y$ if and only if $\text{ran}_{\lhd}(X)<\mathrm{rank}_{\lhd}(Y)$ or if $\text{ran}_{\lhd}(X)=\rho=\mathrm{rank}_{\lhd}(Y)$ then $X\prec_\rho Y$.

\smallskip

The well-order $\prec$ is first-order definable from  $\vec{t}$
and $\vec{\delta}$. Indeed, from  $\vec{t}$
and  $\vec{\delta}$ one defines the functions $f_X$, the relation $\lhd$, the rank levels $R_\rho$, the sets $A_\rho$, the order-isomorphisms $g_\rho$, and finally $\prec$. Since in our case, $\vec{t},\vec{\delta}$ are in $\HOD$ we conclude that $\mathcal{P}(\kappa)\s \HOD$: Let $X\s \kappa$.
Let $\rho$ be the order type of the set of $\prec$-predecessors of $X$. Then $X$ is the
unique subset of $\kappa$ which is the $\rho$-th element of the $\OD$ well-order $\prec$. Since $X$ is a set of ordinals, we conclude that $X\in \HOD$. 
 \end{proof}
 Incidentally, we  proved the following variant of Shelah's theorem:
 \begin{theorem}
        Suppose that   $\kappa$ is a singular cardinal  of uncountable cofinality with   $2^{\cf(\kappa)}<\kappa$ and that for each $\alpha < \kappa$ there is an $\mathrm{OD}$ well-ordering of $\mathcal{P}(\alpha)$. Then,   there is an $\mathrm{OD}$ well-ordering of $\mathcal{P}(\kappa)$.\qed
 \end{theorem}
 \begin{remark}
    Zhixing You has observed using the LLM provider ChatGPT that the conclusion of  Theorem~\ref{thm: compactness} can be obtained replacing $``2^{\cf(\kappa)}<\kappa$" by $``\HOD$ is cardinal-correct" -- in fact,  $``(2^{\cf(\kappa)})^+=(2^{\cf(\kappa)})^{+\HOD}$" suffices: Let us denote $\mu:=(2^{\cf(\kappa)})^+$ and $A_\rho$ be defined as in the proof of Theorem~\ref{thm: compactness}.   $A_\rho$  is $\OD$  and $|A_\rho|< \mu$ hence our assumption yields an injection $g_\rho\colon A_\rho\to \mathcal{P}(\cf(\kappa))$  in $\HOD$. Consider  $g_\rho\circ f_X\colon \cf(\kappa)\to \mathcal{P}(\cf(\kappa))$. Then,  $g_\rho\circ f_X$ can be coded as a subset of $\cf(\kappa)$\footnote{To show this, let $\Gamma\colon \cf(\kappa)\times \cf(\kappa)\to \cf(\kappa)$ be Gödel's pairing function and fix $h\colon \cf(\kappa)\to \mathcal{P}(\cf(\kappa))$. Put $X_h:=\{\Gamma(\alpha,\beta)\mid \alpha < \cf(\kappa),\, \beta\in h(\alpha)\}$. Then, $X_h\in \mathcal{P}(\cf(\kappa))\s \HOD$ and $\Gamma\in \HOD$ (by absoluteness).   $h$ can be decoded inside $\HOD$ from $X_h$ and $\Gamma$ as $h=\{(\alpha,\{\beta\mid \Gamma(\alpha,\beta)\in X_h\})\mid \alpha < \cf(\kappa)\}$, ergo $h\in \HOD$.} and thus it belongs to $H_\kappa^{\HOD}$. The rest of the argument translates  verbatim.
 \end{remark}

\section{The  covering property for $(\HOD, V)$}\label{sec: compactness of covering}
In the previous section we discussed  the possible values of the first place of disagreement, $\Delta(\HOD, V)$. This time we would like to analyze another cardinal characteristic  pertaining to the cover property of the pair $(\HOD, V)$:
\begin{definition}
    Let $(M, N)$ be a pair of inner models with $M\s N$, and let $\mu$ be an $N$-cardinal. Then, we write $\mathscr{C}_{\mu}(M,N)$ for $$\min\{\alpha\in \mathrm{Ord}\mid (M,N)\;\text{does not have the $(\mu^+)^N$-cover property at $\alpha$}\},$$
    provided the above set is non-empty -- otherwise, set $\mathscr{C}_\mu(M,N):=\mathrm{Ord}$. 

    We write $\mathscr{C}(M,N)$ for 
    $$\min\{\alpha\in \mathrm{Ord}\mid (M,N)\;\text{does not have the full cover property at $\alpha$}\},$$
    provided the above set is non-empty -- otherwise, set $\mathscr{C}(M,N):=\mathrm{Ord}$. 
\end{definition}

 We would like to elucidate the status of the following questions:
\begin{question}\label{que: full cover}
What can be the first cardinal $\kappa$ such that $(\HOD, V)$  fails to have the full cover property at $\kappa$? Namely, what can be $\mathscr{C}(\HOD, V)$?
\end{question}

\begin{question}\label{que: partial cover}
  What can be the first cardinal $\kappa$ such that $(\HOD, V)$ fails to have the  $\cf(\kappa)^+$-covering at $\kappa$? 
\end{question}
Easily, any regular cardinal $\kappa$ can be made (via forcing) the first place where $(\HOD,V)$ fails to satisfy the full cover property at $\kappa$. In contrast, no such a cardinal  can violate the $\cf(\kappa)^+$-covering at $\kappa$ between $\HOD$ and $V$. For singulars  of countable cofinality the situation is easy as well.
\begin{proposition}
 Assuming the existence of a measurable cardinal it is consistent for a singular strong limit cardinal of countable cofinality to be the first cardinal $\kappa$ where $(\HOD, V)$ fails to satisfy the $\omega_1$-covering property.
\end{proposition}
\begin{proof}
Suppose that $\kappa$ is a measurable cardinal and $V=g\HOD$. Let $\mathbb{P}$ be Prikry forcing relative to a normal measure on $\kappa$. Since $\mathbb{P}$ is weakly homogeneous, $V=\HOD^{V[G]}$ for any $G\s \mathbb{P}$ generic. However, in $V[G]$, the Prikry sequence converging to $\kappa$ is a countable set in $V[G]$ that cannot be covered by subsets of $\kappa$ in $\HOD^{V[G]}$  of size $\aleph_0$. Since $V[G]_\kappa = V_\kappa$ it follows that $\kappa$ is the first cardinal where the failure of covering takes place.
\end{proof}
How about singular cardinals of uncountable cofinality? As before, the situation becomes much more intriguing and $\mathrm{ZFC}$ itself has a word on it:

\begin{theorem}[Compactness of $\mu^+$-covering]\label{cor: compactness of covering}
    Let $\kappa$ be a singular cardinal with uncountable cofinality such that $2^{\cf(\kappa)}<\kappa$.  Then, $$\mathscr{C}_\mu(\HOD, V)\neq \kappa$$ for all cardinals $\mu\geq \cf(\kappa)$ such that $2^\mu<\kappa$.
\end{theorem}
\begin{proof}
Suppose for the sake of a contradiction  that $\mathscr{C}_{\mu}(\HOD, V)=\kappa$.

\smallskip

The following is the analogue of Theorem~\ref{thm: spinoffGP}:
\begin{claim}
There is a cofinal set $c\s\kappa$, in $\HOD$, with $|c|\leq \mu$.    
\end{claim}
\begin{proof}[Proof of claim]
By our covering hypothesis, every ordinal in the stationary set $(\kappa\cap \cof(\omega))\setminus (\mu+1)$ is $\HOD$-singular, hence \cite[Theorem~3.5]{GolPov} applies, yielding $\cf^{\HOD}(\kappa)<\kappa$.  Let $e\colon \cf^{\HOD}(\kappa)\to \kappa$ be a  cofinal map in $\HOD$ and $\sigma\s \cf^{\HOD}(\kappa)$ a cofinal set of size $\cf(\kappa)$. (It exists as $\cf(\cf^{\HOD}(\kappa))=\cf(\kappa)$.) Then, there is $c'\in \HOD$ of size ${\leq}\mu$ such that $\sigma\s c'\s\cf^{\HOD}(\kappa)$. Finally, let $c:=e[c']$.
\end{proof}

Let $\langle \kappa_i\mid i<\chi\rangle\in \HOD$ be an enumeration of $c$ above. Note that $|\chi|\leq \mu$ and $\cf(\chi)=\cf(\kappa)\geq \omega_1$. Let also $\langle \langle \tau^i_{\xi}\mid \xi <\chi_i\rangle\mid i<\mu\rangle\in \HOD$
be such that $\langle \tau^i_\xi\mid \xi<\chi_i\rangle$ is an enumeration of the $\HOD$-subsets of $\kappa_i$ with $V$-cardinality $\leq \mu$. For each $\sigma\s \kappa$ with $|\sigma|\leq \mu$ look at the trace function $f_\sigma\colon \chi\to \mathrm{Ord}$,
$$f_\sigma(i):=\min\{\xi<\chi_i\mid \sigma\cap \kappa_i\s \tau^i_\xi\}.$$ Given sets $\sigma_0,\sigma_1\in\mathcal{P}_{\mu^+}(\kappa)$ let us define
$$\sigma_0\lhd \sigma_1\;\;\Longleftrightarrow\;\; \{i<\chi\mid f_{\sigma_0}(i)\geq f_{\sigma_1}(i)\}\in \mathcal{J}_{\mathrm{bd}}^{\chi}.$$
As in Theorem~\ref{thm: compactness}, $\lhd$ is a well-founded relation (because $\cf(\chi)\geq \omega_1$) such that for each ordinal $\rho\in \mathrm{Ord}$, $$R_\rho:=\{\sigma\in \mathcal{P}_{\mu^+}(\kappa)\mid \mathrm{rank}_{\lhd}(\sigma)=\rho\}\in \HOD$$  and it has size at most $2^\mu$ by the Erd\"os--Rado theorem.\footnote{Here we use that  $|\chi|\leq \mu$.}

\smallskip

Let $\sigma\in \mathcal{P}_{\mu^+}(\kappa)$ be arbitrary, and let us cover it.  Set $\rho:=\mathrm{rank}_{\lhd}(\sigma)$. It follows that $\bigcup R_\rho$ is in $\HOD$, it covers $\sigma$ and it has size at most $2^\mu<\kappa$. Let $e\colon \lambda\to \bigcup R_\rho$ be the unique order-preserving bijection, which belongs to $\HOD$. Then, $e^{-1}[\sigma]\s \lambda$ is a set of size ${\leq}\mu$ ergo, since $\mathscr{C}_{\mu}(\HOD, V)=\kappa$, there is some $\tau'$ in $\HOD$ such that $e^{-1}[\sigma]\s \tau'$ and $|\tau'|\leq \mu$. It follows that $\sigma\s \tau:=e[\tau']$ and the latter is a set in $\HOD$ of $V$-cardinality at most $\mu$.
\end{proof}
The next theorem completes the discussion regarding Question~\ref{que: full cover}:
\begin{theorem}\label{cor:fullcovering}
    If $\kappa$ is a singular strong limit cardinal with $\cf(\kappa)\geq \omega_1$ then $$\mathscr{C}(\HOD, V)\neq \kappa.$$
\end{theorem}
\begin{proof}
    Suppose for the sake of a contradiction  that $\mathscr{C}(\HOD, V)=\kappa$. Let $\sigma\s\kappa$ be an uncountable offending set. Clearly $|\sigma|<\kappa$, for otherwise $\kappa$ would itself serve as a $\HOD$-cover. Similarly, if $|\sigma|<\cf(\kappa)$ then $\sigma$ is a subset of some $\alpha < \kappa$ and so, as $\mathscr{C}(\HOD, V)=\kappa$,  there is $\tau\in \HOD$ such that $\sigma\s \tau$ with $|\sigma|=|\tau|$. Thus, it must be that $|\sigma|\in [\cf(\kappa),\kappa)$. In that case, since $\kappa$ is strong limit, $2^{|\sigma|}<\kappa$, and since  $\mathscr{C}(\HOD, V)=\kappa$ entails $\mathscr{C}_{|\sigma|^+}(\HOD, V)\geq \kappa$, we get  $\tau\in \HOD$ such that $\sigma\s\tau$ and $|\tau|\leq |\sigma|$ ($\leq|\tau|$). This yields the desired contradiction, thus establishing the corollary.
\end{proof}
There is a variation of the previous theorem where the assumption
\begin{center}
  ``$(\HOD, V)$ has the full cover property at $\alpha$ for all $\alpha < \kappa$''  
\end{center} is replaced by the weaker \begin{center}
``$(\HOD, V)$ has the  $\cf(\lambda)^+$-cover property at $\lambda$ for all cardinals $\lambda < \kappa$''    
\end{center} 
at the expense of assuming the additional  $\mathrm{Card}^{\HOD}\cap \kappa=\mathrm{Card}\cap \kappa$. 

\smallskip

The next theorem completes the discussion around Question~\ref{que: partial cover} assuming cardinal-correctedness of the pair $(\HOD, V)$:
\begin{theorem}
Suppose that $\kappa$ is a strong limit cardinal with $\cf(\kappa)\geq \omega_1$ and the following properties hold:
\begin{enumerate}
    \item $\mathrm{Card}^{\HOD}\cap \kappa=\mathrm{Card}\cap \kappa$.
    \item $(\HOD, V)$ has the  $\cf(\lambda)^+$-cover property at $\lambda$ for all  $\lambda < \kappa$.\footnote{If $\lambda$ is regular then the assumption is trivially true.}
\end{enumerate}
Then, $(\HOD, V)$ has the full cover property at every ordinal $\alpha \leq \kappa$.
\end{theorem}
\begin{proof}
By virtue of Theorem~\ref{cor:fullcovering} it suffices to check that $(\HOD, V)$ has the full cover property at every ordinal $\alpha <\kappa$. Towards a contradiction, suppose that this is not the case, and let $\alpha < \kappa$ be the first ordinal and $\theta$ be the first cardinal witnessing that $(\HOD, V)$ does not have the full cover property at $\alpha$. This is witness by an offending set $\sigma\s \alpha$ with $|\sigma|=\theta$.  
\begin{claim}
    $\alpha$ is a cardinal.

    In particular, $(\HOD, V)$ has the $\cf(\alpha)^+$-cover property at $\alpha$.
\end{claim}
\begin{proof}[Proof of claim]
 Suppose $\alpha$ was not a cardinal, and let $e\colon |\alpha|\to \alpha$ be a bijection in $\HOD$ (which exists, by cardinal correctedness). Then, $e^{-1}(\sigma)\s|\alpha|$ admits a $\HOD$-cover of size $\theta$. Since $e\in \HOD$, pulling this cover via $e$ we get a $\HOD$-cover of $\sigma$ by a set of size $\theta$. Contradiction.   
\end{proof}
We distinguish four cases: 

\smallskip
$\br$ If $\theta=\alpha$ then $\alpha$ itself serves as a $\HOD$-cover.

$\br$ If $\theta<\cf(\alpha)$ then there is $\beta < \alpha$ such that $\sigma\s \beta$. However, $(\HOD, V)$ has the full cover property at $\beta$ (by mininality of $\alpha$), ergo there is $\tau\in \HOD$ with $|\sigma|=|\tau|$, yielding the desired contradiction.

$\br$ If $\theta=\cf(\alpha)$ then the previous claim disposes with this case.

$\br$ Finally, assume that $\theta\in (\cf(\alpha),\alpha)$.  Let $\langle \alpha_i\mid i<\cf(\alpha)\rangle$ be an increasing cofinal sequence in $\alpha$. By the minimality of $\alpha$, for each $i<\cf(\alpha)$,
$$\sigma\cap \alpha_i\s \tau_i$$
for some $\tau_i\in \HOD$ with $|\tau_i|=|\sigma\cap \alpha_i|$. Clearly, $|\tau_i|\leq \theta$.

\smallskip

Let $\mathcal{C}=\mathcal{P}_{\theta^+}(\alpha)\cap \HOD$ and $e\colon \chi \to \mathcal{C}$ be a bijection. By our assumption that $\kappa$ is strong limit, $\chi<\kappa$. Ergo, $\{e^{-1}(\tau_i)\mid i<\cf(\alpha)\}$ is a subset of $\chi$ of $\cf(\alpha)$. Since $(\alpha,\theta)$ was the minimal pair, and $\cf(\alpha)<\theta$,  there is a cover $\tau'\in \HOD$ with  cardinality $\cf(\alpha)$. Hence,
$\tau :=\bigcup_{\xi\in \tau'}e(\xi)$ belongs to $\HOD$, it covers $\sigma$ and has size $\cf(\alpha)\cdot\theta =\theta$.  
\end{proof}

\section{The optimality of Magidor's covering theorem}\label{sec: optimality of Magidors}
The aim of this section is to prove a compactness theorem for the covering property that does not assume the inner model $M$ being equal to $\HOD$. Afterwards we shall prove the optimality of this theorem via forcing.

The proof of our more general compactness theorem  mimics Magidor's argument for Theorem~\ref{thm: Magidor}, as presented in \cite[Theorem~2.33]{AbrahamMagidor}. We thus recommend our readers to have a copy of this latter reference close to them while reading the next result.   All the standard PCF tools and notions used in the proof are exposed in greater detail in \cite[\S2]{AbrahamMagidor}.
\begin{proposition}
    [Local Magidor Covering]\label{lemma: local Magidor}
    Suppose \(\kappa\) is a singular cardinal of uncountable cofinality \(\mu\) and  \(M\) is an inner model satisfying:
    \begin{enumerate}
        \item\label{clause: cardinalcorrectness} $ \mathrm{Card}^M\cap (\kappa^{+M}+1)=\mathrm{Card}\cap (\kappa^++1)$ and {$\cf^M(\kappa)=\mu$}.        \item\label{clause: GCH} $\{\gamma<\kappa\mid M\models \gamma^\mu\leq \gamma^+\}$ is stationary in $\kappa$.
        \item\label{clause: coverbelow}           $\mathscr{C}_\mu(M,V)\geq \kappa$.    \end{enumerate}
    Then, $\mathscr{C}_{\mu}(M,V)\geq \kappa^+$.\footnote{In particular, $M$ has the $\mu^+$-cover property at $\kappa$.}\end{proposition}
\begin{proof}
    Let $X\s \kappa^+$ be a set with $|X|\leq \mu$. We construct  $Y\in M$ such that $X\s Y\s \kappa^+$ and $|Y|\leq \mu$.     Through the proof, unless otherwise specified, by \emph{cardinality} we shall mean cardinality in the sense of the universe $V$.

   Set $\sigma:=\sup(X)$. Since $|X|\leq \mu$ it follows that $\sigma < \kappa^+$. By cardinal-correctness (Clause~\eqref{clause: cardinalcorrectness}) this is the same as saying that $\sigma<\kappa^{+M}$. Fix a bijection $e\colon \kappa' \rightarrow \sigma$ in $M$ for some $\kappa'\leq \kappa$. Then, $e^{-1}[X]\s \kappa$ and is a set of cardinality $\leq\mu$.     If we were able to prove the existence of some $X^*\in M$ with $e^{-1}[X]\s X^*\s \kappa$ and $|X^*|\leq \mu$ then we would have that $Y:=e[X^*]$ is the desired covering set from $M$.  So, everything reduces to the problem of covering $V$-subsets of $\kappa$ of cardinality ${\leq}\mu$ by means of $M$-subsets of $\kappa$  of cardinality ${\leq \mu}$. In a slight abuse of notation for the rest of the proof we shall keep denoting this subsets of $\kappa$ of cardinality ${\leq}\mu$ by $X$.

    \smallskip

    Let $X\s \kappa$ be a set in $V$ with size ${\leq}\mu$. If $X$ was bounded in $\kappa$ then, by covering below $\kappa$ (Clause~\eqref{clause: coverbelow}) there would be a covering set in $M$ of size ${\leq}\mu$.

    So, let us assume that $X$ is unbounded in $\kappa$.  Since $\kappa$ has uncountable cofinality in $M$, we can use the proof of Shelah's representation theorem (see \cite[Theorem~2.23]{AbrahamMagidor} and the comments in \cite[Theorem~2.33]{AbrahamMagidor}) inside  $M$ to  manufacture a scale $\langle f_\xi \mid \xi <\kappa^+\rangle$     in a product $(\prod_{i<\mu}\kappa_i^+)^M$ where $\langle \kappa_i\mid i<\mu\rangle$ is a club on $\kappa$ consisting of $M$-singular cardinals with $\mu<\kappa_0$. Moreover, this $M$-scale can be defined so that for every limit ordinal $\xi < \kappa^+$ there is a club $E_\xi  \s\xi$ with $\mathrm{otp}(E_\xi)=\cf(\xi)$ and 
    $$f_\xi=^*\sup\{f_\zeta\mid \zeta\in E_\xi\}.$$
Here, $=^*$ denotes equality modulo bounded.
    While all of this is done inside $M$, Magidor argues (see \cite[Theorem~2.23]{AbrahamMagidor}) that by shrinking this $M$-scale to a club subset of $\kappa$, one obtains a $V$-scale. More formally, utilizing  \cite[Lemma~2.19]{AbrahamMagidor} we deduce that $(\star)_\theta$ holds\footnote{For the definition of $(\star)_\theta$ see \cite[Definition~2.8]{AbrahamMagidor}.} for all regular cardinal $\theta < \kappa$.  Hence,   $\langle f_\xi\mid \xi < \kappa^+\rangle$ has an exact upper bound $h$ for which the set $\{i<\mu\mid \cf(h(i))<\theta\}$ is   bounded in $\mu$, for all $\theta < \kappa$. From this one can find a club $\langle \kappa_{i_j}\mid j<\mu\rangle \s \langle \kappa_i\mid i<\mu\rangle$ for which $h(i_j)=\kappa_{i_j}^+$ (i.e., $\kappa_{i_j}^{+M}$, by cardinal correctedness) for all $j<\mu$ (see \cite[Claim~2.24]{AbrahamMagidor}). This implies that the restriction of the $M$-scale to this sub-club,   $\langle f_\xi \restriction \{i_j\mid j<\mu\}\mid \xi < \kappa^+\rangle$,  remains  a $V$-scale (i.e., $<^*$-cofinal) in the product   $(\prod_{j<\mu} \kappa_{i_j}^+)^V$. Crucially, the unrestricted members of this sequence (i.e., the $f_\xi$'s) all belong to $M$.
We use this scale to cover $X$.
    \smallskip

    By Clause~\eqref{clause: GCH},  stationary many members $\kappa_{i_j}$ of the sequence satisfy $$\text{$M\models ``{\kappa_{i_j}}^\mu=\kappa_{i_j}^+$".}$$     As the subsequent proof will make clear, we can assume   that the above cardinal-arithmetic configuration holds for all $j<\mu$.

    \smallskip

    For each $j<\mu$, $X\cap \kappa_{i_j}$ is a  subset of $\kappa_{i_j}$ of size $\leq \mu$. Thus, Clause~\eqref{clause: coverbelow} applies, yielding a cover $X_{j}\in M$, $X\cap \kappa_{i_j}\s X_j\s \kappa_{i_j}$  with $|X_j|\leq \mu$. 
    
    Let $\langle Z^j_\eta \mid \eta <\kappa_{i_j}^+\rangle$ be an enumeration of all $M$-subsets of $\kappa_{i_j}$ of cardinality $\leq \mu$.     For this we use   Clause~\eqref{clause: cardinalcorrectness} and \eqref{clause: GCH} and our choice of the  $\kappa_{i_j}$'s.

    \smallskip

    Let $F\in (\prod_{j<\mu}\kappa_{i_j}^+)^V$ be the map given by the formula $$\text{$F(j):=``$The first index $\eta <\kappa^+_{i_j}$ such that $X_j=Z^j_\eta$''.}\footnote{Note that $F$ may not be a member of  $M$. }$$

Since $\langle f_\xi\restriction\{i_j\mid j<\mu\}\mid \xi<\kappa^+\rangle$ is a scale in $V$,  there is an index $\xi<\kappa^+$ for which  $F<^* f_\xi\restriction \{i_j\mid j<\mu\}$.  To simplify the  forthcoming argument, we shall assume that  $F(j)<f_\xi(i_j)$ for all $j<\mu$.

\smallskip

For each $i<\mu$ let $g_i\colon f_\xi(i)\to \kappa_{i}$ be an injection in $M$, which exists by cardinal-correctness below $\kappa$ (i.e., by Clause~\eqref{clause: cardinalcorrectness}). Critically, since $f_\xi \in M$, the sequence of bijections $\langle g_i\mid i<\mu\rangle$ belongs to $M$, as well.

\smallskip

Look at the set of ordinals  $$\{g_{i_j}(F(j))\mid j<\mu\}.$$

Note that this set may be external to the inner model $M$. 

For each $j\in \mu\cap \cf(\omega)$ there is some  $l<j$ such that $g_{i_j}(F(j))<\kappa_{i_l}$ because $\langle \kappa_{i_j}\mid j<\mu\rangle$ is a club on $\kappa$. Applying Fodor's lemma in $V$ we find a stationary set $S\s \mu \cap \cof(\omega)$ and an index $j^*<\mu$ such that $g_{i_j}(F(j))<\kappa_{i_{j^*}}$ for all $j\in S$. Thus  $\{g_{i_j}(F(j))\mid j\in S\}\s \kappa_{i_{j^*}}$ and it has cardinality $\leq \mu$. While both $S$ and $\{g_{i_j}(F(j))\mid j\in S\}$ may only belong to the universe $V$ we can still invoke the $\mu^+$-covering property at $\kappa_{i_{j^*}}$ (i.e., Clause~\eqref{clause: coverbelow}) to find a covering set $U\s \kappa_{i_{j^*}}$ in $M$ with $|U|\leq \mu$.

   \smallskip

  Now we are ready to cover $X$. Define $Y:=\bigcup\{Z^j_\eta\mid  j<\mu\,\wedge\, \eta\in g^{-1}_j(U)\}.$  Clearly, $X\s Y\s \kappa$. Also $Y\in M$, as both $U$ and $\langle g_j\mid j<\mu\rangle$ belong to $M$.

\begin{claim}
   $|Y|\leq \mu$. 
\end{claim}
\begin{proof}[Proof of claim]
 By our choice, $|Z^j_\eta|\leq \mu$. We have $\mu$-many $j$'s and for each of them we have at most $\mu$-many $\eta$ such that $\eta\in g_j^{-1}(U)$ (because $g_j$ is an injection and $U$ has size ${\leq}\mu$). So, all in all, the above is a union of $\mu$-many sets of cardinality $\mu$ so, by regularity of $\mu$,  $|Y|\leq \mu$. 
\end{proof}
We are done with the proof of the lemma.
\end{proof}

Let us now prove the optimality of Proposition~\ref{lemma: local Magidor} (and hence of Magidor's theorem Theorem~\ref{thm: Magidor}) by showing that the GCH assumption (Clause~\eqref{clause: GCH}) cannot be removed from Proposition~\ref{lemma: local Magidor}. The next general lemma gives a natural way to violate the covering property:
\begin{lemma}\label{lemma: cover}
    Suppose that $\kappa$ is a singular cardinal such that $2^{\cf(\kappa)} < \kappa$ and $(M, V)$ has the $\cf(\kappa)^+$-cover property at $\kappa$. Then, $|(2^\kappa)^{M}|=2^\kappa$
\end{lemma}
\begin{proof}
Let $\langle \kappa_\alpha\mid \alpha < \cf(\kappa)\rangle$ be a continuous, increasing, cofinal sequence in $\kappa$. For each $\alpha <\cf(\kappa)$ let $\langle Z^\alpha_\xi\mid \xi < 2^{\kappa_\alpha}\rangle$ be an injective enumeration of $\mathcal{P}(\kappa_\alpha)$ in $V$.   For each $X\s \kappa$ consider the map $F_X\in \prod_{\alpha < \cf(\kappa)}2^{\kappa_\alpha}$ defined by $F_X(\alpha):=``$The unique $\xi<2^{\kappa_\alpha}$ such that $X\cap \kappa_\alpha = Z^\alpha_\xi$''. Clearly, the map $X\mapsto F_X$ is injective, ergo $2^\kappa\leq |\prod_{\alpha < \cf(\kappa)} 2^{\kappa_\alpha}|$. We shall see that the latter is less than or equal to $|(2^\kappa)^{M}|$. To show this we use the $\cf(\kappa)^+$-cover.

For each $X\s\kappa$, $\mathrm{range}(F_X)\in \mathcal{P}_{\cf(\kappa)^+}(\kappa)$ ergo there is $Y\in M$ with $|Y|\leq \cf(\kappa)$ such that $\mathrm{range}(F_X)\s Y\s \kappa$. Thus, $F_X\in \prod_{\alpha <\cf(\kappa)} Y$ and $$\textstyle |\prod_{\alpha <\cf(\kappa)} Y|=\cf(\kappa)^{\cf(\kappa)}=2^{\cf(\kappa)}.$$
With this, we obtain the following bounds
$$\textstyle |\prod_{\alpha < \cf(\kappa)} 2^{\kappa_\alpha}|\leq |\mathcal{P}_{\cf(\kappa)^+}^{M}(\kappa)|\cdot 2^{\cf(\kappa)} = |(2^\kappa)^{M}|\cdot 2^{\cf(\kappa)}\leq  |(2^\kappa)^{M}|,$$
because $2^{\cf(\kappa)} < \kappa \leq |(2^\kappa)^{M}|$.
\end{proof}

The next consistency result proves the optimality of Magidor's theorem:
\begin{theorem}\label{thm:optimality of Magidor}
    Assume the $\mathrm{GCH}$ holds and that $\vec{E}=\langle E_\alpha\mid \alpha<\omega_1\rangle$ is a Mitchell increasing sequence of $(\kappa,\kappa^{++})$-extenders\footnote{For the abstract definition of a $(\kappa,\theta)$-extender we refer the reader to \cite[\S26]{Kan}.}     such that, for each $\alpha <\omega_1$, $\kappa^{++\mathrm{Ult}(V,E_\alpha)}=\kappa^{++}$. Then, in some cardinal-preserving generic extension $W$ there is an inner model $V$ such that:
    \begin{enumerate}
         \item $(V,W)$ compute cofinalities in the same way and $\cf^{W}(\kappa)=\omega_1$
        \item The set $\{\alpha <\kappa\mid V\models 2^\alpha=\alpha^{++}\}$ contains a club in $\kappa$.
        \item $(V,W)$ has the $\omega_1$-covering property. (Moreover, $({}^\omega \mathrm{Ord})^V=({}^\omega\mathrm{Ord})^W$.)
        \item $(V,W)$ have full covering at $\alpha$ for all $\alpha <\kappa$.
        \item $(V,W)$  fails to have $\omega_2$-covering at $\kappa$. (Moreover, $(2^\kappa)^V=\kappa^+$ and $(2^\kappa)^W=\kappa^{++}$.)
    \end{enumerate}
    In particular, the $\mathrm{GCH}$ (or even $\mathrm{SSH}$) assumption cannot be removed from Magidor's theorem.\footnote{During the presentation of our results at the Hebrew University Set Theory seminar, M. Gitik informed us that the same was observed by  M. Magidor in the early 90's.  Magidor mentioned that he does not remember having proved the theorem.  Regardless, it seems that this may have been know by some experts, albeit it was never published nor  part of the folklore of the field.}
\end{theorem}

\begin{remark}
    The method used in the proof of the above theorem is general enough so as to give models $(V,W)$ where $\kappa$ has cofinality $\mu$ for any prescribed regular cardinal $\mu<\kappa$ and $(2^\kappa)^{W}$ is arbitrarily large.

    \end{remark}

The blanket assumptions for the rest of the section are as follows:

\begin{setup}\label{ref: setup}
\rm{We assume  $\mathrm{GCH}$ holds and fix $\vec{E}= \langle E_\alpha \mid \alpha< \omega_1\rangle$  a Mitchell increasing sequence of $(\kappa,\kappa^{++})$-extenders. We denote by  $\mathbb{P}$ the  Extender Based Radin forcing as defined by Merimovich in \cite{Merimovich} and fix a generic filter $G\s \mathbb{P}$.\footnote{We shall assume the reader is acquainted with Merimovich's forcing.} We also fix  a regular cardinal $\Theta$ such that  $2^{2^{|\mathbb{P}|}}<\Theta$. Utilizing  $\mathrm{GCH}$, we can construe an elementary submodel $M\prec H_\Theta$, with:
\begin{itemize}
    \item $\{\mathbb{P}({\vec{E}}),\vec{E},\kappa^{++}\}\cup V_\kappa\s M$,
    \item $|M|=\kappa^{+}$, ${}^\kappa M\s M$,
    \item and $\chi:=M\cap \kappa^{++}\in \kappa^{++}$ and $\cf(\chi)=\kappa^+$.
\end{itemize}  }
\end{setup}

\smallskip

 Our goal is to show  that 
$(V[G\cap M], V[G])$
is a pair of models witnessing the statement of the above theorem. 

\begin{notation}
Hereafter we denote $\mathbb{P}:=\mathbb{P}(\vec{E})$ and $\mathbb{P}_M:=\mathbb{P}\cap M$.   
\end{notation}

\begin{lemma}\label{lemma: residue}
    There is a  residue map $\restriction_M\colon \mathbb{P}\to \mathbb{P}\cap M$. \end{lemma}
\begin{proof}
Let us define $ \restriction_M$ as follows: Given a condition $p=p_{\leftarrow}{}^\smallfrown \langle f, T\rangle$ in $\mathbb{P}$, $$p\restriction_M:=p_{\leftarrow}{}^\smallfrown \langle f\restriction M, T\restriction M\rangle,$$
where $T\restriction M$ is defined as  $\{\langle \nu_0\restriction M,\dots, \nu_{n-1}\restriction M\rangle\mid \langle \nu_0,\dots, \nu_{n-1}\rangle\in T\}.$

The first observation is that $p\restriction_M$ is a condition in $\mathbb{P}_M$. This is a routine checking using the fact that $V_\kappa\s M$ and that ${}^\kappa M\s M$. 

\smallskip

Let us show that $\restriction_M$ is a residue map. Let $r\leq_{\mathbb{P}_M} p\restriction_M$, $r\in M$, and let us show that $r$ is $\leq_{\mathbb{P}}$-compatible with $p$. First, by elementarity, $r\leq_{\mathbb{P}_M} p\restriction_M$ is the same as $r\leq_{\mathbb{P}} p\restriction_M$. By definition, there is $\vec\nu\in T\restriction M$ such that $$\text{$r_\leftarrow \leq (p\restriction_M\cat \vec\nu)_{\leftarrow}$ and $r_\rightarrow\leq^*_{\mathbb{P}} (p\restriction_M\cat \vec\nu)_\rightarrow$}.\footnote{Here $\leq$ denotes the product ordering associated to the `lower' extender based forcing.}$$
To simplify notations let us assume that $\vec\nu =\langle \nu\rangle$.

Then, $r_\leftarrow$ decomposes as $r_{\leftarrow\leftarrow}{}^\smallfrown r_{\leftarrow\rightarrow}$ and similarly $(p\restriction_M\cat \langle \nu\rangle )_{\leftarrow}$ decomposes as $(p\restriction_M\cat \langle \nu\rangle)_{\leftarrow\leftarrow}{}^\smallfrown (p\restriction_M\cat \langle \nu\rangle)_{\leftarrow\rightarrow}$. Notice that $(p\restriction_M\cat \langle \nu\rangle)_{\leftarrow\leftarrow}$ is just $p_{\leftarrow}$ due to the definition of the map $\restriction_M$. Set $p'_0$ to be $r_{\leftarrow\leftarrow}$.

Now, let us analyze $(p\restriction_M\cat \langle \nu\rangle)_{\leftarrow\rightarrow}$. By definition of $\cat \langle \nu\rangle$, this is  $$\langle (f\restriction M)\downarrow \langle \nu\rangle, (T\restriction M)\downarrow \langle \nu\rangle\rangle;$$
more explicitly:
\begin{itemize}
    \item $(f\restriction M)\downarrow \langle \nu\rangle$ is the function with domain $\{\nu(\bar \alpha)\mid o(\bar\alpha)\neq 0\}$ and
    $$\left((f\restriction M)\downarrow \langle \nu\rangle\right)(\nu(\bar\alpha)):=(f\restriction M)(\bar \alpha)\restriction k_\nu=f(\bar \alpha)\restriction k_\nu,$$
    where $k_\nu:=\min\{l<|f(\bar\alpha)|\mid\forall l\leq i<|f(\bar\alpha)|\, (o(f_i(\bar\alpha))<o(\nu(\bar\alpha))) \}$.
    \item $(T\restriction M)\downarrow \langle \nu\rangle$ is the projected tree $$\{\langle \mu_0\downarrow \nu, \dots, \mu_{n-1}\downarrow\nu\rangle\mid \vec\mu\in (T\restriction M)\},$$
    which is the same as $\{\langle \mu_0\downarrow \nu, \dots, \mu_{n-1}\downarrow\nu\rangle\mid \vec\mu\in T\}.$
\end{itemize}
Let $\eta\in T$ be a $\dom(f)$-object such that $\eta\restriction M=\nu$. 

We define $p'_1$ to be $\langle f^{r_{\leftarrow\rightarrow}}\cup (f\downarrow \eta), S_1\rangle$ where $S_1$ is the intersection of the pullbacks to  $\dom(f^{r_{\leftarrow\rightarrow}})\cup \dom(f\downarrow\eta)$ of the trees $T^{r_{\leftarrow\rightarrow}}$ and $(T\downarrow\eta)$. 

Finally, let $p'_2$ to be $\langle f^{r_{\rightarrow}}\cup (f\cat \eta), S_2\rangle$ where $S_2$ is defined analogously.

Finally, we set $p'=p'_0{}^\smallfrown p'_1{}^\smallfrown p'_2$. It follows that this is a condition in $\mathbb{P}$ and that $p'$ witnesses compatibility of $r$ and $p$.
\end{proof}
The next follows from the previous lemma and Proposition~\ref{cor: submodel}:
\begin{corollary}
    $G\cap M$ is a $\mathbb{P}_M$-generic filter.\qed
\end{corollary}

Given a condition $p\in \mathbb{P}$ the \emph{Factor Lemma} for the Extender Based Radin forcing \cite[p. 461]{Merimovich} ensures that $\mathbb{P}/p$ is isomorphic to a two-step product
$$\mathbb{P}_{\vec e}/p_0\times \mathbb{P}/p_1$$
where $p_0$ is an initial segment of $p_{\leftarrow}$, $\vec{e}\in V_\kappa$ is a  sequence of extenders and $\mathbb{P}_{\vec e}$ is the corresponding Extender Based Radin forcing.  Let $\pi_{\vec{e}}\colon \mathbb{P}/p\to \mathbb{P}_{\vec e}/p_0$ be the projection induced from the previous isomorphism.  
\begin{lemma}
     $\pi_{\vec{e}}\colon \mathbb{P}_M/p\restriction_M \to \mathbb{P}_{\vec{e}}/p_0$ is a projection and the diagram 
$$
\begin{tikzcd}
    \mathbb{P}/p \arrow[r, "{\pi_{\vec{e}}}"] \arrow[d, "\restriction_M"'] 
     & \mathbb{P}_{\vec{e}}/p_0  \\
\mathbb{P}_M/p\restriction_M \arrow[ru, "{\pi_{\vec{e}}\restriction \mathbb{P}_M}" '] &  
\end{tikzcd}
$$
is commutative.
\end{lemma}
\begin{proof}
$\pi_{\vec e}\restriction \mathbb{P}_M$ is clearly order-preserving. Let $q\in \mathbb{P}_M$ be with  $q\leq p\restriction_M$ and $r\leq_{\mathbb{P}_{\vec e}} \pi_{\vec{e}}(q)$. Say $q=\langle q_i\mid i<\ell(q)\rangle$. Since $\restriction_M$ is a residue map, there is $p'\in \mathbb{P}$ such that $p'\leq p, q$ and $p'\restriction M=q$. In particular, $p'_{\leftarrow}=q_{\leftarrow}$. Since $\vec{e}$ for some $i_0<\ell(p')$, $\langle p'_i\mid i<i_0\rangle\in \mathbb{P}_{\vec e}$ there is  $i^*<\ell(q)$ such that  $\langle q_i\mid i<i^*\rangle \in \mathbb{P}_{\vec{e}}$. Hence, $\pi_{\vec{e}}(q)=\langle q_i\mid i<i^*\rangle$. Then, $r\leq_{\mathbb{P}_{\vec{e}}} \langle q_i\mid i<i^*\rangle$. Letting $q'=r{}^\smallfrown \langle q_i\mid i\in [i^*,\ell(q))\rangle$ we obtain a condition in $\mathbb{P}_M$, stronger than $q$ such that $\pi_{\vec{e}}(q')=r$.
\end{proof}

In particular, if $G_{\vec e}$ is the $\mathbb{P}_{\vec e}\,$-generic filter induced by $\pi_{\vec e}$ and $G$, 
\begin{equation}\label{eq: inclusion between models}
\tag{*}  V[G_{\vec e}]\s V[G_M]\s V[G].  
\end{equation}

\begin{lemma}
    For each $\alpha < \kappa$, $\mathcal{P}^{V[G]}(\alpha)=\mathcal{P}^{V[G_M]}(\alpha)$.
\end{lemma}
\begin{proof}
    This follows from the usual Radin-like analysis establishing that every bounded subset of $\kappa$ in $V[G]$ is a member of $V[G_{\vec e}]$. Indeed, combining this with (*) above we obtain the desired equality.
\end{proof}
\begin{remark}
    It is crucial here to note that for every $\alpha<\omega_1$, $\chi=(\kappa^{++})^{M_{E_\alpha\restriction\chi}}$. 
    Indeed, since $\bar{M}$ is closed under $\kappa$ sequences, every function  $f:[\kappa]^n\to\kappa$ is a member of $\bar{M}$. Also, $E_\alpha\restriction\chi\in \bar{M}$. It follows that $(j_{E_\alpha\restriction\chi}(V_{\kappa}))^{\bar{M}}=j_{E_\alpha\restriction\chi}(V_{\kappa})$. Since $\chi=(\kappa^{++})^{\bar{M}}$, $\bar{M}\models E_\alpha\restriction\chi\text{ is a }(\kappa,\kappa^{++})$-extender.
    Hence $\bar{M}$ thinks that the ultrapower by $E_\alpha\restriction\chi$ computes $\kappa^{++}$ the same as in $\bar{M}$. It follows that $\chi=(\kappa^{++})^{\bar{M}}=(\kappa^{++})^{(M_{E_\alpha\restriction\chi})^{\bar{M}}}=(\kappa^{++})^{M_{E_\alpha\restriction\chi}}$.  
\end{remark}
\begin{lemma}
    $\mathbb{P}_M$ is isomorphic\footnote{Both with respect to the regular and the Prikry order.} to the Extender Based Radin $\mathbb{P}(\vec{F})$ where $\vec{F}$ denotes the sequence of truncated extenders $\langle E_\alpha\restriction \chi\mid \alpha<\omega_1\rangle$. \end{lemma}
 \begin{proof}
Let $\pi\colon M\to \bar{M}$ be the Mostowski collapse map. By elementarity, $$\text{$M\models ``\mathbb{P}(\vec{E})$ is the Extender Based Radin relative to $\vec{E}$".}$$ Thus, 
$$\bar{M}\models \text{$``\pi(\mathbb{P}(\vec{E}))$ is the Extender Based Radin relative to $\pi(\vec{E})$"}.$$
Observe that $\crit(\sigma)=\chi$ and $\sigma(\chi)=\kappa^{++}$ where $\sigma:=\pi^{-1}$.  This yields $\pi(E_\xi)=\{\pi(E)_{\xi,a}\mid a\in [\chi]^{<\omega}\}$ which equals $\{\pi(E_{\xi,a})\mid a\in [\chi]^{<\omega}\}=\{E_{\xi,a}\mid a\in [\chi]^{<\omega}\}=E_\xi\restriction \chi$.\footnote{To justify $\pi(E_{\xi,a})=E_{\xi,a}$  we have used that $\pi(X)=X$ for every $X\in [\kappa]^n$ and $n<\omega$.} As a result we have that $\pi(\vec{E})=\langle E_\xi\restriction \chi\mid \xi<\omega_1\rangle.$ Ergo, $\bar{M}$ thinks that  $``\pi(\mathbb{P}(\vec{E}))$ is the Extender Based Radin relative to $\vec{F}$''. But this definition is absolute between $\bar{M}$ and the universe because $\bar{M}$ is closed under $\kappa$-sequences of its elements (as so is $M$ and $\crit(\sigma)>\kappa$). In other words, $\pi(\mathbb{P}(\vec{E}))=\mathbb{P}(\vec{F})$. But then $\pi\colon \mathbb{P}(\vec{E})\cap M \to \mathbb{P}(\vec{F})$ establishes the desired isomorphisms: both with respect to $\leq_{\mathbb{P}(\vec{E})}$ and $\leq^*_{\mathbb{P}(\vec{E})}$.
 \end{proof}

 \begin{lemma}\label{lemma: power function}     $\mathbb{P}(\vec{F})$ forces $``2^\kappa=\kappa^+$". 
     
     In particular, the same configuration is forced by $\mathbb{P}_M$.
 \end{lemma}
 \begin{proof}
  To justify that $\mathbb{P}(\vec{F})$ forces $2^\kappa=\kappa^+$ we will show that every $\mathbb{P}(\vec{F})$-name for a subset of $\kappa$ can be 'reduced' to a name for a subposet of $\mathbb{P}(\vec{F})$ of size $\kappa$. Since there are  $|\mathbb{P}(\vec{F})|^\kappa=\kappa^+$ many of such subposets $\mathbb{Q}$ and at most $|\mathbb{Q}|^{\kappa\cdot\kappa}=\kappa^+$ many $\mathbb{Q}$-nice names for subsets of $\kappa$ we will conclude that $$\one\forces_{\mathbb{P}(\vec{F})}\text{$`` 2^\kappa=\kappa^+$"}.$$

    For the scope of this  lemma write $\mathbb{P}:=\mathbb{P}(\vec{F})$. Let $\dot{a}$ be a $\mathbb{P}$-name and $p\in \mathbb{P}$ forcing $``\dot{a}\s \kappa$". Let $N\prec H_\Theta$ be with $|N|=\kappa$, ${}^{<\kappa} N\s N$, $N\cap \kappa^+\in \kappa^+$ and $\{p,\mathbb{P},\dot{a}\}\cup V_\kappa\s N$. By properness of $\mathbb{P}$, there is $q\leq^* p$ that is $(\mathbb{P},N)$-generic (see \cite[Lemma 4.13]{Merimovich}). For each $\alpha < \kappa$, denote $$D_\alpha :=\{r\in \mathbb{P}\mid r\parallel_{\mathbb{P}}\check{\alpha}\in\dot{a}\}.$$
    Clearly, this is dense open and it belongs to $N$, and thus $$q\forces_{\mathbb{P}} \check{D}_\alpha\cap \check{N}\cap \dot{G}\neq \emptyset.$$
   Define the $(\mathbb{P}\cap N)$-name $\dot{b}=\{(\check{\alpha},r)\mid \alpha <\kappa,\,r\in \mathbb{P}\cap N,\, r\forces_{\mathbb{P}}\check{\alpha}\in \dot{a}\}$.

  A routine checking shows that   $q\forces_{\mathbb{P}}\dot{b}=\dot{a}$.

   \smallskip

   We have thereby showed that for each $\mathbb{P}$-name $\dot{a}$ for a subset of $\kappa$,
   $$E_{\dot{a}}=\{r\in \mathbb{P}\mid \exists \mathbb{Q}\in [\mathbb{P}]^{\kappa}\,\exists \dot{b}\in V^{\mathbb{Q}}\, (r\forces_{\mathbb{P}}\dot{a}=\dot{b})\}$$
   is dense. Thus, given any $V$-generic $G\s \mathbb{P}$, $$\mathcal{P}(\kappa)^{V[G]} \s \{\dot{b}_G\mid \exists \mathbb{Q}\in [\mathbb{P}]^{\kappa}\;(\dot{b}\, \text{is a $\mathbb{Q}$-nice name})\}.$$
   But as argued above there are at most $\kappa^+$ such nice names, ergo right-hand-side union has size $\kappa^+$.
 \end{proof}

 \begin{lemma}
$({}^\omega \mathrm{Ord})^{V[G]}=({}^\omega \mathrm{Ord})^{V[G_M]}$.
 
    In particular,  $(V[G_ M],V[G])$ has the $\omega_1$-covering property.
 \end{lemma}
 \begin{proof}
     Working in $V$, let $p\in G$ be a condition and $\dot{f}$  be a $\mathbb{P}$-name such that $p\forces_{\mathbb{P}}\dot{f}\colon \omega\rightarrow \mathrm{Ord}$. We will show that $\dot{f}_G\in V[G_M]$. 
\smallskip

 An elementary submodel $\mathcal{M}\prec H_\Theta$ is \emph{adequate} if $|\mathcal{M}|=\kappa$, $\mathcal{M}^{<\kappa}\s \mathcal{M}$, $\mathcal{M}\cap \kappa^+\in \kappa^+$ and $p, \dot{f}, \mathbb{P}\in \mathcal{M}$. By work of Merimovich, $\mathbb{P}$ is a $\kappa$-proper forcing notion \cite[Lemma 4.13]{Merimovich}. In particular, the set
$$\{q\leq_{\mathbb{P}} p\mid \exists \mathcal{M}\, (\text{$\mathcal{M}$ is adeduate and $q$ is $\langle \mathcal{M}, \mathbb{P}\rangle$-generic})\}$$
is dense below $p.$ 

Let $p^*\in G$ belonging to the above set and $\mathcal{M}$ an adequate elementary submodel witnessing it. For each $n<\omega$, consider the dense open set of conditions deciding $\dot{f}(n)$; more explicitly, we are interested on the  set
$$D_n:=\{q\in \mathbb{P}\mid \exists\alpha\in \mathrm{Ord}\, (q\forces_{\mathbb{P}}\dot{f}(n)=\alpha)\}.$$
Clearly, $D_n\in \mathcal{M}$ for all $n<\omega$. So, by $\langle \mathcal{M}, \mathbb{P}\rangle$-genericity, for each $n<\omega$,
$$p^*\forces_{\mathbb{P}}\check{D}_n\cap \dot{G}\cap \check{\mathcal{M}}\neq \emptyset.$$
Since $p^*\in G$, in fact $D_n\cap G\cap \mathcal{M}$ is non-empty for all $n<\omega$. 

For each $n<\omega$, let $q\in G\cap \mathcal{M}$ deciding the value of $\dot{f}(n)$. By elementarity of $\mathcal{M}$, there is $\alpha\in \mathcal{M}$ such that $q\forces_{\mathbb{P}}\dot{f}(n)=\alpha.$ All in all, this shows that 
$$\text{range}(\dot{f}_G)\s \mathcal{M}.$$
Working in $V$, let $\Phi\colon \mathcal{M}\leftrightarrow \kappa$ be a bijection and consider $f^*$ the composition $\Phi\circ \dot{f}_G\colon \omega\rightarrow \kappa$. We claim that $f^*\in V[G_M]$, which yields $\dot{f}_G\in V[G_M]$.

Indeed, since $\cf^{V[G]}(\kappa)=\omega_1$ the function $f^*$ is bounded in $\kappa$; say that $\lambda<\kappa$ is such a bound. Let $q\leq p^*$ be a condition in $G$ forcing $\dot{f}^*\colon \check{\omega}\to \check{\lambda}$.

By the usual factoring lemma of Radin-like forcings we have that $\mathbb{P}/q$ is isomorphic to a product $(\mathbb{P}_{\vec{e}}/q_0)\times (\mathbb{P}/q_1)$ where the latter has a $\leq^*$-ordering that is more than $(\lambda^+)^V$-closed. {Combining this with the Prikry property one can show that the set}
$$\{r\leq_{\mathbb{P}}q\mid \exists \dot{g}\in V^{\mathbb{P}_{\vec{e}}}\, (r\forces_{\mathbb{P}}\dot{g}=\dot{f}^*)\}\footnote{Here, in the forced sentence, we are identifying $\dot{g}$ with its standard $\mathbb{P}$-name.}$$
is dense below $q$. 

Denote by $G_{\vec{e}}$ the $\mathbb{P}_{\vec{e}}$-generic filter obtained after projecting $G$ via the standard projection between $\mathbb{P}$ and $\mathbb{P}_{\vec{e}}$. Thus,  $f^*\in V[G_{\vec{e}}]\s V[G_M]$. 
 \end{proof}

 \begin{lemma}
     $V[G_M]$ is cofinality-correct. 
 \end{lemma}
 \begin{proof}
For a condition $p=\langle f, T\rangle \in \mathbb{P}^*$ define its projection to the normal measure by recursion as follows: For $\langle f, T\rangle\in \mathbb{P}^*$ denote $$\pi_{\{\kappa\}}(\langle f, T\rangle)=\langle f\restriction\{\kappa\}, T\restriction\{\kappa\}\rangle.$$ In general, for a condition $p=\langle p_i\mid i\leq n\rangle\in \mathbb{P}$, define  by recursion $$\pi_{\{\kappa\}}(p):=\pi_{\{\kappa\}}(\langle p_i\mid i<n\rangle){}^\smallfrown\pi_{\{\kappa\}}(p_n).$$
As it turns, $\pi_{\{\kappa\}}$ defines a projection between $\mathbb{P}$ and its image, this latter being  the usual Radin forcing defined with respect to the Mitchell increasing sequence of normal measures $\langle E_\alpha(\{\kappa\})\mid \alpha <\omega_1\rangle$ on $\ob(\{\kappa\})$. 

Note that the diagram 
$$
\begin{tikzcd}
    \mathbb{P} \arrow[r, "{\pi_{\{\kappa\}}}"] \arrow[d, "\restriction_M"'] 
     & \mathbb{P}_{\{\kappa\}}  \\
\mathbb{P}_M \arrow[ru, "{\pi_{\{\kappa\}}\restriction \mathbb{P}_M}" '] &  
\end{tikzcd}
$$
commutes and that 
$\pi_{\{\kappa\}}\restriction \mathbb{P}_M$ is  a projection. 

In particular, we get the inclusions $V[\pi_{\{\kappa\}}[G]]\s V[G_M]\s V[G]$. 

The usual Radin-like analysis show that $V[G]$ and $V[\pi_{\{\kappa\}}[G]]$ compute cofinalities alike, ergo so do $V[G_M]$ and $V[G]$.
 \end{proof}

 \begin{lemma}
      $(V[G_M],V[G])$ satisfy full covering at $\alpha$ for all $\alpha <\kappa$.
 \end{lemma}
 \begin{proof}
     This follows  from $\mathcal{P}(\alpha)^{V[G_M]}=\mathcal{P}(\alpha)^{V[G]}$ for all $\alpha <\kappa$. For every $\alpha <\kappa$ there is an extender sequence $\vec{e}\in V_\kappa$ with $\mathcal{P}^{V[G_{\vec{e}}]}(\alpha)=\mathcal{P}^{V[G]}(\alpha)$ but the former is contained in $V[G_M]$ because of \eqref{eq: inclusion between models} on page~\pageref{eq: inclusion between models}.
 \end{proof}

  \begin{lemma}\label{lemma: failure of covering}
     $(V[G_M],V[G])$ does not satisfy $\omega_2$-covering at $\kappa$.
 \end{lemma}
 \begin{proof}
 This follows from Lemma~\ref{lemma: cover} using that $(2^{\kappa})^{V[G_M]}<(2^\kappa)^{V[G]}$.
 \end{proof}

     \begin{remark}
       As a bonus result, let us observe that   the pair $(V[G_M], V[G])$ has the full approximation property below $\kappa$ but it fails to have the  $\omega_1$-approximation at $\kappa$: Let $C_{\bar\alpha}\s \kappa$ be in $V[G]\setminus V[G_M]$. Then, $C_{\bar\alpha}\cap x\in V[G_M]$ for all $x\in V[G_M]$ with cardinality less than  $\omega_1$ because $C_{\bar\alpha}\cap x$ is bounded below $\kappa$ and as a result is a member of $V[G]_\kappa = V[G_M]_\kappa$. 

       Our method also gives an optimal failure of Thei's \emph{Scale Property} \cite{SebaPhD} at a singular cardinal of uncountable cofinality. 
     \end{remark}

     \section{Open questions}\label{sec: open questions}
Theorem~\ref{thm: compactness} showed that a singular cardinal $\kappa$ of uncountable cofinality with $2^{\cf(\kappa)}<\kappa$ cannot be the first place of disagreement between $\HOD$ and the universe. The proof  made a crucial use \cite[Theorem~3.5]{GolPov} which in turn only assumes $\kappa$ to be a singular cardinal of uncountable cofinality. This raises a question -- to what extent is the assumption   $2^{\cf(\kappa)}<\kappa$ necessary?

\begin{question}
    Suppose that $\kappa$ is a singular cardinal of uncountable cofinality. Is it consistent (starting from large cardinals) that $\Delta(\HOD, V)=\kappa$?
\end{question}

In addition, what happens if we relax the degree of agreement between $\HOD$ and $V$ by just requiring the cardinality of the power set functions $\lambda \mapsto 2^\lambda$ of $V$ and $\HOD$ to agree? Should we still expect compactness?
\begin{question}\label{que: agreement 2kappa}
    Suppose that $\kappa$ is a  singular strong limit cardinal  of uncountable cofinality, that $\{\lambda < \kappa\mid (2^\lambda)^{\HOD}=2^\lambda\}$ is stationary and $(2^\kappa)^{\HOD}$ is a cardinal. Must $(2^\kappa)^{\HOD}=2^\kappa$?
\end{question}
A natural attempt to answering this question would be to code most (but not all) of the subsets of $\lambda < \kappa$ from $V[G_M]$ into the power-set function pattern and show that the $\HOD$ of the resulting generic extension is a subclass of $V[G_M]$, hence a model of GCH at $\kappa$. The obstacle with this approach is ensuring that the automorphisms of $\mathbb{P}$ fixes the name for the coding poset (this would yield an automorphism of $\mathbb{P}\ast \dot {\mathrm{Code}}$). Unfortunately this argument does not quite work because the natural automorphism between  cones of $\mathbb{P}$ by two  incompatible conditions must necessarily swap forcing information below $\kappa$ and as a result this automorphism will move the name for the coding poset $\dot{\mathrm{Code}}$. Since in $V[G_M]$ the various instances for the failure of GCH are witnessed by the existence of long scales it is natural to ask whether or not coding this scales into HOD actually encodes into this model the scales that live on the top cardinal $\kappa$. Thus, we ask:
\begin{question}
Suppose that $\HOD$ is cofinality-correct and $\kappa$ is a strong limit singular cardinal with $\cf(\kappa)\geq \omega_1$. Let $\langle \kappa_\alpha\mid \alpha < \omega_1\rangle\in \HOD$ be a club on $\kappa$. Assuming that  $\{\alpha <\kappa\mid \mathrm{tcf}(\prod^{\HOD}_{\beta < \alpha}\kappa_\beta^+)=\mathrm{tcf}(\prod_{\beta < \alpha}\kappa_\beta^+)\}$ is stationary, must $\mathrm{tcf}(\prod^{\HOD}_{\alpha < \cf(\kappa)}\kappa_\alpha^+)=\mathrm{tcf}(\prod_{\alpha < \cf(\kappa)}\kappa_\alpha^+)$.
\end{question}
Finally, it is natural to ask if the additional cardinal-correctedness assumption utilized in Theorem~\ref{thm:cflambda+cover} is actually necessary:
\begin{question}
    Is it possible to remove the cardinal-correctedness assumption from Theorem~\ref{thm:cflambda+cover}?
\end{question}

\section*{Acknowledgements}
The results of this paper are part of the \emph{SQuaRE}  project \emph{``Compactness and ultrafilter combinatorics''} sponsored by the  American Institute of Mathematics. The authors are very grateful to AIM for their hospitality and generous support. Cummings was supported by NSF grant DMS2054532. Goldberg was supported by National Science Foundation under Grant No. 2401789.
Hayut was supported by the Israel Science Foundation, Grants no. 1967/21 and 3469/25. Poveda 
was supported by a 2026 Leonardo Grant for Scientific Research and Cultural Creation from the BBVA Foundation \includegraphics[width=0.1\linewidth]{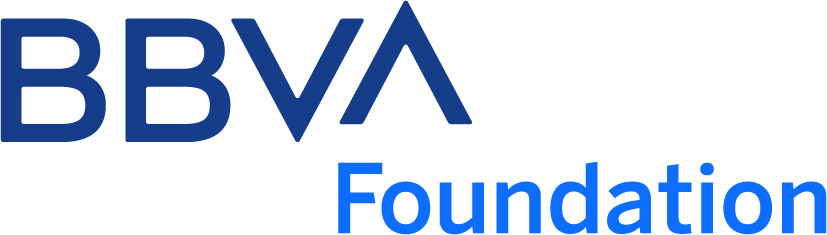} \footnote{The BBVA Foundation accepts no responsibility for the opinions, statements and contents included in the project and/or the results thereof, which are entirely the responsibility of the authors.}, by project  PID2023-147428NB-I00 from the Spanish Government by the \emph{Alexander von Humboldt Foundation} \includegraphics[width=0.1\linewidth]{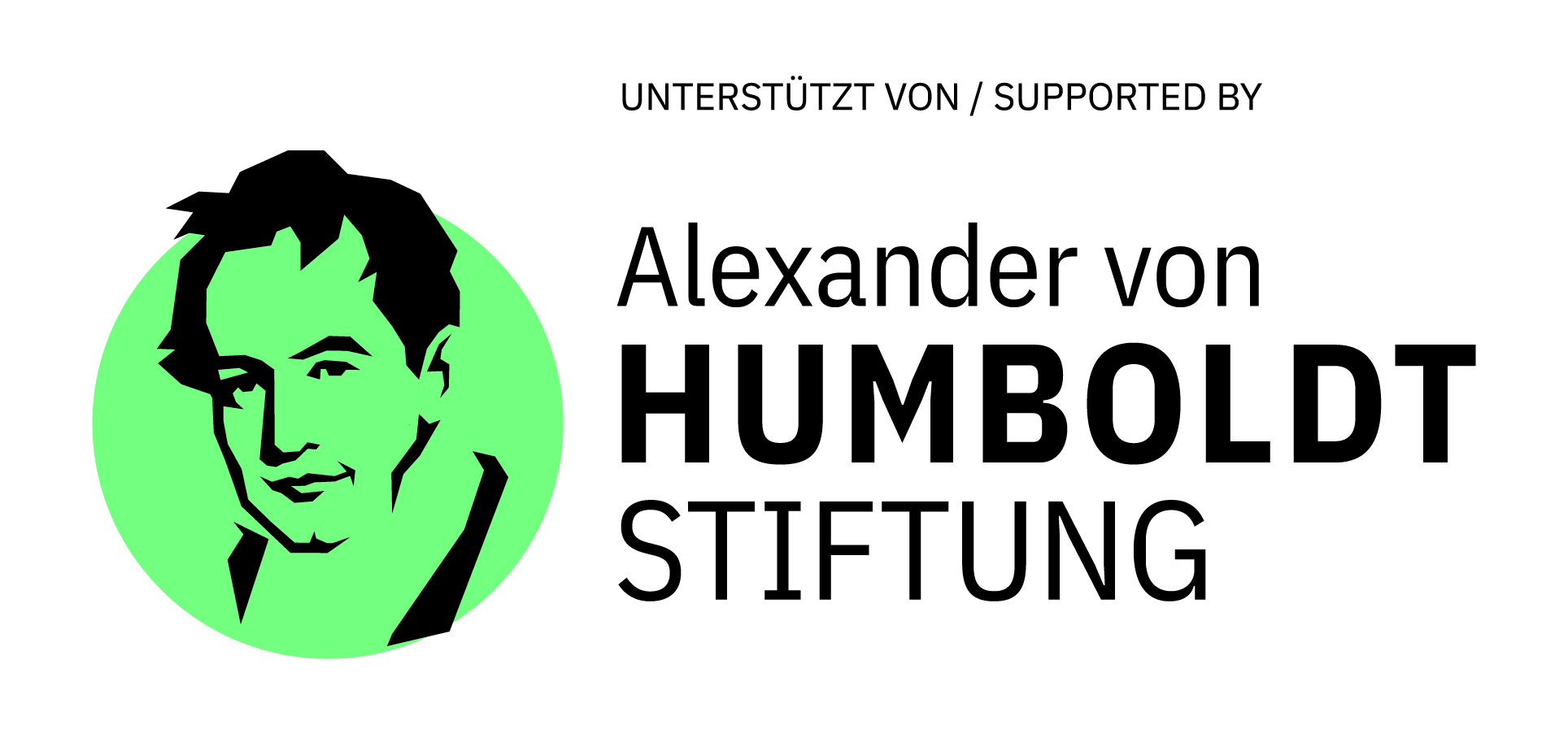}.

       \bibliographystyle{alpha}
\bibliography{citations}
\end{document}